\documentclass[reqno]{amsart}

\usepackage{xcolor}
\usepackage{graphicx}
\usepackage{hyperref}
\usepackage{mathtools}
\usepackage{mathrsfs}
\usepackage{amsmath}
\usepackage{bm}
\usepackage{tikz-cd}
\usepackage{amsthm}

\usepackage{doi}

\usepackage[utf8]{inputenc}

\newtheorem{theorem}{Theorem}[section]
\newtheorem{lemma}[theorem]{Lemma}

\newtheorem{corollary}[theorem]{Corollary}
\newtheorem{claim}[theorem]{Claim}

\theoremstyle{definition}
\newtheorem{remark}[theorem]{Remark}

\newtheorem{convention}[theorem]{Convention}
\newtheorem{notation}[theorem]{Notation}

\usepackage{amssymb}
\renewcommand{\leq}{\leqslant}
\renewcommand{\geq}{\geqslant}
\newcommand{\eps}{\epsilon}
\newcommand{\ca}{\mathcal}
\newcommand{\wt}{\widetilde}

\def\C{\text C}
\def\N{\mathrm N}
\def\L{\text L}

\def\F{\mathbf F}
\def\dim{\text{dim}\:}
\def\dimT{\emph{dim}\:}
\def\Z{\text Z}
\def\GL{\text{GL}}
\def\GLT{\emph{GL}}
\def\Gal{\text{Gal}}
\def\SU{\text{SU}}
\def\PSU{\text{PSU}}
\def\SUT{\emph{SU}}
\def\GU{\text{GU}}
\def\GUT{\emph{GU}}
\def\Sp{\text{Sp}}
\def\PSp{\mathrm{PSp}}
\def\SpT{\emph{Sp}}
\def\SL{\mathrm{SL}}
\def\PSL{\mathrm{PSL}}
\def\SLT{\emph{SL}}
\def\O{\mathrm{O}}
\def\OT{\emph{O}}
\def\SO{\text{SO}}
\def\SOT{\emph{SO}}
\def\GO{\mathrm{O}}
\def\P{\mathbf P}
\def\C{\text C}
\def\Om{\Omega}
\def\POm{\mathrm P\Omega}

\let\emptyset\varnothing
\def\wt{\widetilde}

\newcommand{\gen}[1]{\ensuremath{\langle #1\rangle}}

\usepackage{accents}
\newcommand\thickbar[1]{\accentset{\rule{.4em}{.8pt}}{#1}}

\makeatletter
\def\blfootnote{\xdef\@thefnmark{}\@footnotetext}
\makeatother

\usepackage{todonotes}
\newcommand{\eilidh}[1]{\todo[color=red!40]{#1}{}}
\newcommand{\daniele}[1]{\todo[color=green!40]{#1}{}}

\begin{document}

\title{On the probability of generating invariably a finite simple group}

\author{Daniele Garzoni}
\address{School of Mathematical Sciences, Tel Aviv University, Tel Aviv 69978,
Israel}
\email{danieleg@mail.tau.ac.il}

\author{Eilidh McKemmie}
\address{Department of Mathematics, Rutgers University, Piscataway, NJ, 08854, USA}
\email{emck@rutgers.edu}

\begin{abstract}
Let $G$ be a finite simple group. In this paper we consider the existence of small subsets $A$ of $G$ with the property that, if $y \in G$ is chosen uniformly at random, then with high probability $y$ invariably generates $G$ together with some element of $A$. We prove various results in this direction, both positive and negative. As a corollary, we prove that two randomly chosen elements of a finite simple group of Lie type of bounded rank invariably generate with probability bounded away from zero. Our method is based on the positive solution of the Boston--Shalev conjecture by Fulman and Guralnick, as well as on certain connections between the properties of invariable generation of a group of Lie type and the structure of its Weyl group.
\end{abstract}

\maketitle

\section{Introduction}
\label{section_intro}

\blfootnote{The first author was partially supported by a grant of the Israel Science Foundation No. 702/19, and has received funding from the European Research Council (ERC) under the European Union’s Horizon 2020 research and innovation programme (grant agreement No. 850956). The second author was partially supported by the NSF grant DMS 1901595.}

Let $G$ be a finite group. Given a subset $A=\{x_1, \ldots, x_t\}$ of $G$, we say that $A$ \textit{invariably generates} $G$ if $\gen{x_1^{g_1}, \ldots, x_t^{g_t}}=G$ for every choice of $g_1, \ldots, g_t \in G$. We write in this case $\gen{x_1, \ldots, x_t}_I=G$. This concept was introduced by Dixon with motivations from computational Galois theory; see \cite{Dix92} for details. 

For a subset $A$ of $G$, denote by $\P_{\text{inv}}(G,A)$ the probability that, if $y \in G$ is chosen uniformly at random, there exists $x \in A$ such that $\gen{x,y}_I=G$. In case $A=\{x\}$, we will write $\P_{\text{inv}}(G,x)$ instead of $\P_{\text{inv}}(G,\{x\})$.

We state our first result, which in most cases will be asymptotically superseded by subsequent theorems. In Subsection \ref{subsection_context} we will provide more context for these theorems, also in relation to ``classical'' generation.

\begin{theorem}
\label{main_theorem_bounded_away_zero}
There exists an absolute constant $\epsilon >0$ such that every nonabelian finite simple group $G$ contains an element $x\in G$ satisfying $\P_{\emph{inv}}(G,x) \geq \epsilon$.
\end{theorem}

Guralnick--Malle \cite{GM} and Kantor--Lubotzky--Shalev \cite{KLS} independently proved that every finite simple group is invariably generated by two elements. Theorem \ref{main_theorem_bounded_away_zero} is a strengthening of this result. 

In view of the aforementioned results in \cite{GM} and \cite{KLS}, we may assume that $|G|$ is large in order to prove Theorem \ref{main_theorem_bounded_away_zero}.


In light of Theorem \ref{main_theorem_bounded_away_zero}, one would like to find many elements $x\in G$ with the property that $\P_{\text{inv}}(G,x)$ is bounded away from zero uniformly. For groups of Lie type of bounded rank different from $G_2(3^a)$, one can take almost all elements. In the following statement, the \textit{untwisted rank} of $G$ denotes the rank of the ambient simple algebraic group.



\begin{theorem}
\label{main_theorem_strengthening_bounded_rank}
Let $G$ be a finite simple group of Lie type of untwisted rank $r$ defined over $\F_q$, or over $\F_{q^2}$ if $G$ is a Suzuki or Ree group, 
and assume $G\not\cong G_2(q)$ when $3\mid q$. Then,
\[
\P_{\emph{inv}}(G,x)\geq c/r+O(r^r/q)
\]
for an absolute constant $c>0$ and for a proportion of elements $x \in G$ of the form $1-O(r^r/q^a)$, where $a=1/2$ if $G\cong \PSL_2(q)$ and $a=1$ otherwise.
\end{theorem}

In order to avoid confusion, we point out that with $u=O(z)$ we mean that $|u|\leq Cz$ for some constant $C$ (so there is no assertion on the sign of $u$).

\begin{remark}
\label{r: suzuki}
    The parameter ``$q$'' of Theorem \ref{main_theorem_strengthening_bounded_rank} is the one appearing in Table \ref{main_table}. In fact, this parameter admits a conceptual definition in terms of algebraic groups and Steinberg endomorphisms; see Section \ref{section_bounded rank} for details. We stress that $q$ is an integer, except for Suzuki and Ree groups, where $q$ is not an integer but $q^2$ is an integer. See Remark \ref{r_notation}(iii) for  further comments regarding our notation for Suzuki and Ree groups.
\end{remark}

In Theorem \ref{main_theorem_strengthening_isogeny_types} we will give an explicit value for the constant $c$ appearing in Theorem \ref{main_theorem_strengthening_bounded_rank}. The groups $G=G_2(3^a)$ are genuine exceptions; we will see that $\P_{\text{inv}}(G,x)=0$ for roughly half of the elements $x\in G$, but we can show that $\P_{\text{inv}}(G,x)\geq 1/6+O(1/q)$ for the remaining  elements (see Theorem \ref{theoremg2order}).
By considering separately the error term in the case  $G\cong \PSL_2(q)$, we get the following immediate consequence.

\begin{theorem}
\label{main_theorem_two_random_elements_bounded_rank}
Let $G$ be a finite simple group of Lie type of untwisted rank $r$ defined over $\F_q$, or over $\F_{q^2}$ if $G$ is a Suzuki or Ree group. Let $x_1,x_2\in G$ be chosen uniformly at random. Then,
\[
\P(\gen{x_1,x_2}_I=G) \geq c/r+O(r^r/q)
\]
for an absolute constant $c>0$.
\end{theorem}

Of course, in the previous two theorems we are thinking of $r$ fixed, and $q\rightarrow \infty$. We will review the history of Theorem \ref{main_theorem_two_random_elements_bounded_rank} in Subsection \ref{subsection_bounded_rank}. An interesting goal is to seek sharp bounds in Theorem \ref{main_theorem_two_random_elements_bounded_rank}. Although in this paper we do not pursue this goal, 
in Theorem \ref{theorem_elaboration} we will obtain a formula of the type
\[
\P(\gen{x_1,x_2}_I=G)=f(r)+d(r)/q.
\]
The main term $f(r)$ is very explicit, depending only on the Weyl group of $G$, and can be computed essentially in an algorithmic way; it should be possible to compute it precisely for all exceptional groups.


Eberhard--Ford--Green \cite{EFG} and the second author \cite{McK} showed that, for alternating groups of large degree and for groups of Lie type of large rank over large fields, $\P(\gen{x_1,x_2,x_3}_I=G)$ tends to zero. (Conjecturally, in groups of Lie type there should be no restriction on the field size.) Therefore, it is not possible to extend Theorem \ref{main_theorem_strengthening_bounded_rank} to the other families of finite simple groups -- not even for a proportion of elements bounded away from zero. 



In Theorems \ref{main_theorem_bounded_away_zero} and \ref{main_theorem_strengthening_bounded_rank} we have bounded $\P_{\text{inv}}(G,x)$ away from zero. Next, we would like to get probabilities approaching $1$. With this purpose, we consider more elements simultaneously. In most cases, only a few elements are needed, which is in stark contrast to the remaining cases, where the asymptotic statement does not hold even if we take the whole group.


\begin{theorem}
\label{main_theorem_absolute_constant_2}
Let $G$ be a nonabelian finite simple group.
\begin{enumerate}
    \item Assume $G$ is of Lie type and $G\not\cong G_2(q)$ when $3\mid q$. Then, there exists a subset $A_b$ of $G$ such that $|A_b|$ and $\P_{\emph{inv}}(G,A_b)$ satisfy the bounds in Table \ref{main_table}.
    \item Assume $G$ is alternating or classical. Then, there exists a subset $A_\ell$ of $G$ such that $|A_\ell|$ and $\P_{\emph{inv}}(G,A_\ell)$ satisfy the bounds in Table \ref{main_table}.
    \item Assume $G=G_2(q)$ with $3\mid q$, or $G=\emph{PSp}_{2m}(q)$ with $q$ even and $m$ sufficiently large, or $G=\emph P\Omega_{2m+1}(q)$ with $q$ odd and $m$ sufficiently large. Then $\P_{\emph{inv}}(G,G)$ satisfy the bounds in Table \ref{main_table}.
\end{enumerate}
\end{theorem}

We make some remarks on the statement of Theorem \ref{main_theorem_absolute_constant_2} and on the notation and conventions in Table \ref{main_table}.

\begin{remark}
\label{r_notation}
\begin{itemize}
    \item[(i)] In Table \ref{main_table}, ``$r$'' denotes the untwisted rank of $G$. Moreover, it is understood that every nonempty entry in the fourth column of the table applies to all subsequent empty lines.
    
    \item[(ii)] If $G=\PSU_n(q)$ we assume $n\geq 3$; if $G=\PSp_{2m}(q)$ we assume $m\geq 2$; if $G=\POm^\pm_{2m}(q)$ we assume $m\geq 4$; if $G=\POm_{2m+1}(q)$ we assume $m\geq 3$. These assumptions are valid due to the existence of exceptional isomorphisms between certain low rank classical groups. Moreover, in Table \ref{main_table}, for each $G$ we further exclude the finitely many values of $(n,q)$ or $(m,q)$ for which $G$ is not simple; see, e.g. \cite[Proposition 2.9.1]{KL} (this is not crucial, since the statements have no content for bounded $|G|$).
    
    \item[(iii)] If $G=^2\hspace{-0.4em}B_2(q^2)$ or $G=^2\hspace{-0.4em}F_4(q^2)$ then $q^2=2^{2a+1}$ with $a\geq 1$. If $G=^2\hspace{-0.4em}G_2(q^2)$ then $q^2=3^{2a+1}$ with $a\geq 1$.    A common notation in the literature is $^2\!B_2(q)$ with $q=2^{2a+1}$ and $a\geq 1$, etc. Of course, one notation is obtained from the other by performing a change of variables, and $^2\!B_2(8)$, say, denotes the same group in both cases. Our choice is convenient since groups of Lie type behave uniformly with respect to the parameter ``$q$'', which is not an integer for Suzuki and Ree groups; see Section \ref{section_bounded rank} for more details. We note that the same convention is adopted, for instance, in \cite{MT}; see \cite[Table 22.1, p. 193]{MT}.
    
    \item[(iv)] In $A_b$, ``$b$'' stands for ``bounded'', and in $A_\ell$, ``$\ell$'' stands for ``large''. This refers to the rank of the groups; the reason of this choice should be clarified by looking at the bounds in Table \ref{main_table}.

\end{itemize}
\end{remark}

Theorem \ref{main_theorem_absolute_constant_2} presents a strong dichotomy; it is worth stating this separately. Consider the following groups: 
\begin{equation}
\label{eq_cane}
G=\begin{cases}
    G_2(q) & \text{with $3|q$}\\
    \PSp_{2m}(q) & \text{with $q$ even}\\
    \POm_{2m+1}(q)  & \text{with $q$ odd}.
\end{cases}
\end{equation}

\begin{corollary}
\label{main_theorem_absolute_constant}
Let $G$ be a nonabelian finite simple group with $|G|\to \infty$.
\begin{itemize}
    \item[(i)] Assume $G\not\cong G_2(q)$ when $3|q$, and assume that if $G=\PSp_{2m}(q)$ with $q$ even or $G=\POm_{2m+1}(q)$ with $q$ odd then $q\to \infty$. Then, there exists $A \subseteq G$ of size at most $6$ such that $\P_{\emph{inv}}(G,A)$ tends to $1$ as $|G| \rightarrow \infty$.
    \item[(ii)] Assume $G$ is as in \eqref{eq_cane}, and assume that if $G=\PSp_{2m}(q)$ with $q$ even or $G=\POm_{2m+1}(q)$ with $q$ odd then $q$ is bounded. Then, $\P_{\emph{inv}}(G,G)$ remains bounded away from $1$ as $|G|\to \infty$. (Equivalently, $\P_{\emph{inv}}(G,y)=0$ for a proportion of elements $y\in G$ bounded away from zero.)
\end{itemize}
\end{corollary}


\begin{table}
\small
\centering
\caption{Bounds given in the statement of Theorem \ref{main_theorem_absolute_constant_2}. See Remark \ref{r_notation} for notation and conventions used in the table.
}
\label{main_table}       
\begin{tabular}{llll}
\hline\noalign{\smallskip}
$G$ & Conditions & Size of $A_b$ or $A_\ell$  & Bounds \\
\noalign{\smallskip}\hline\noalign{\smallskip}
&  & $|A_b|$ &\\
$^2\!B_2(q^2)$  &  &  $2$ & $\P_{\text{inv}}(G,A_b)\geq1-O(r^r/q)$\\
$^2\!G_2(q^2)$  &  & $2$ & \\
$G_2(q)$ & $3\nmid q$  & $2$ & \\
$^3\!D_4(q)$ &   & $2$ &\\
$^2\!F_4(q^2)$ &   & $2$ &\\
$F_4(q)$ & $q$ odd  & $2$ &\\
$F_4(q)$ & $q$ even  & $6$ &\\
$E_6(q)$ &  & $2$ &\\
$^2\!E_6(q)$ &  & $2$ &\\
$E_7(q)$ &  &  2 &\\
$E_8(q)$ &  & $2$ &\\
$\PSL_n(q)$ & &  $2$ &\\
$\PSU_n(q)$ & & $2$ &\\
$\PSp_{2m}(q)$ & $m$ even, $q$ odd & $2$ &\\
$\PSp_{2m}(q)$ & $m$ odd, $q$ odd & $3$ &\\
$\PSp_{2m}(q)$ & $q$ even & $4$ &\\
$\POm_{2m+1}(q)$ & $q$ odd & $2$ &\\
$\POm_{2m}^-(q)$ & & $2$ &\\
$\POm_{2m}^+(q)$ & $m$ odd & $2$ &\\
$\POm_{2m}^+(q)$ & $m$ even & $4$ &\\
\noalign{\smallskip}\hline
 &  & $|A_\ell|$ &\\
$A_n$ & & $1$ & $\P_{\text{inv}}(G,A_\ell)\geq 1 - O(n^{-0.08})$\\
$\PSL_n(q)$ & &  $1$ & $\P_{\text{inv}}(G,A_\ell)\geq 1 - O(r^{-0.005})$\\
$\PSU_n(q)$ & & $1$ &\\
$\PSp_{2m}(q)$ & $q$ odd & $1$ &\\
$\POm_{2m}^-(q)$ & & $1$ &\\
$\POm_{2m}^+(q)$ &  & $1$ &\\
$\PSp_{2m}(q)$ & $q$ even & $2$ & $\P_{\text{inv}}(G,A_\ell)\geq 1 - 6/q+O(r^{-0.005})$\\
$\POm_{2m+1}(q)$ & $q$ odd & $2$ & \\
\noalign{\smallskip}\hline\noalign{\smallskip}
\hline\noalign{\smallskip}
$G_2(q)$ & $3 \mid q$ & & $\P_{\text{inv}}(G,G)=1/2+O(1/q)$\\
$\PSp_{2m}(q)$ & $q$ even, $m$ large &  & $\P_{\text{inv}}(G,G)\leq 1 - 1/4q^3$\\
$\POm_{2m+1}(q)$ & $q$ odd, $m$ large &  & $\P_{\text{inv}}(G,G)\leq 1 - 1/6q$\\
\noalign{\smallskip}\hline\noalign{\smallskip}
\end{tabular}
\end{table}


Corollary \ref{main_theorem_absolute_constant} follows from Theorem \ref{main_theorem_absolute_constant_2} by setting $A=A_\ell$ for alternating groups, $A=A_b$ for exceptional groups, and $A=A_b\cup A_\ell$ for classical groups.

Of course, every case in which the size of $A_\ell$ is equal to $1$ represents a strengthening of Theorem \ref{main_theorem_bounded_away_zero}. Combining this with Theorem \ref{main_theorem_strengthening_bounded_rank}, we see that in most cases asymptotically we can do better than Theorem \ref{main_theorem_bounded_away_zero}. The improvement is complementary: while there is no sensible analogue of Theorem \ref{main_theorem_strengthening_bounded_rank} in large rank, here we cannot have $|A_b|=1$ in bounded rank (see Lemma \ref{classical_bounded_rank_must_have_at_least_two_elements}). We note finally that the size of $|A_\ell|$ is sharp in every case (Lemmas \ref{orthogonal_odd_|A|_at_least_2} and \ref{symplectic_even_|A|_at_least_2}), and that there are cases in which we need $|A_b|\geq 4$ (see Lemma \ref{four_elements_might_be_needed}; note that $|A_b|\leq 4$ unless $G=F_4(2^a)$). Here, with ``sharp'' we mean that if we choose a set $Y$ of smaller size, then $\P_{\text{inv}}(G,Y)$ remains bounded away from $1$ as the relevant parameters grow.

\subsection{Context: Theorem \ref{main_theorem_two_random_elements_bounded_rank}}
\label{subsection_bounded_rank}

The first result of the flavour of Theorem \ref{main_theorem_two_random_elements_bounded_rank} was obtained by Dixon \cite{Dix92}. Dixon showed that $O((\log n)^{1/2})$ random elements of $S_n$ invariably generate $S_n$ with probability tending to $1$ as $n\rightarrow \infty$. Luczak--Pyber \cite{LP} 
showed that $O(1)$ random elements of $S_n$ invariably generate with probability bounded away from zero. (One cannot approach $1$ with a bounded number of elements, since a random permutation has a fixed point with probability approaching $1-1/e$.) The exact value of $O(1)$ turned out to be four: Pemantle--Peres--Rivin \cite{PPR} proved that four elements are enough, while Eberhard--Ford--Green \cite{EFG} showed that three are not. The same results hold for alternating groups.
The second author \cite{McK} extended these results to classical groups of large rank, proving that four random elements of a finite simple classical group of rank $r$ defined over $\F_q$ invariably generate with probability bounded away from zero if $r$ and $q$ are large enough, and three random elements invariably generate with probability tending to zero as $r,q\to \infty$.

Theorem \ref{main_theorem_two_random_elements_bounded_rank} addresses the case of groups of Lie type of bounded rank, which therefore nearly finishes the problem of invariable generation of finite simple groups by randomly chosen elements -- only the case of classical groups of large rank over small fields remains open. (As in the case of symmetric groups, using a bounded number of elements the probability cannot approach $1$; this follows from results by Fulman--Guralnick \cite{FG}, and it is summarized for instance in \cite[Corollary 5.7]{KLS}.)

Probabilistic invariable generation has been studied also for general finite groups. Confirming a conjecture of Kowalski and Zywina \cite{KZ}, Lucchini \cite{Luc} proved that, for every finite group $G$, picking $O(|G|^{1/2})$ random elements is sufficient, on average, in order to generate $G$ invariably. Lucchini--Tracey \cite{LT} showed that the above bound can be improved to $(1+\eps)|G|^{1/2}+ O_\eps(1)$ for every $\eps >0$.

\subsection{Context: Theorems \ref{main_theorem_bounded_away_zero}, \ref{main_theorem_strengthening_bounded_rank} and \ref{main_theorem_absolute_constant_2}}
\label{subsection_context}
It is convenient to visualize things as follows. The \textit{generating graph} $\Gamma(G)$ of $G$ is the graph whose vertices are the nonidentity elements of $G$, with $x$ and $y$ adjacent if $\gen{x,y}=G$. The first author \cite[Subsection 1.3]{Gar} defined a similar graph $\Lambda_e(G)$: the vertex set is the same, and $x$ and $y$ are adjacent if $\gen{x,y}_I=G$. In this language, Theorem \ref{main_theorem_bounded_away_zero} says that $\Lambda_e(G)$ contains large stars when $G$ is simple.

For $x \in G$, let $\P(G,x)$ denote the probability that, if $y \in G$ is random, then $\gen{x,y}=G$: this is the ``classical'' version of our $\P_{\text{inv}}(G,x)$. Set then $P^-(G)=\text{min}\{\P(G,x) : 1 \neq x \in G\}$. Guralnick and Kantor \cite{GK} showed that $P^-(G)>0$ for every finite simple group $G$, i.e., $\Gamma(G)$ has no isolated vertices. Moreover, in \cite{GLSS}, following results from \cite{GK2} and \cite{LSh2}, the behaviour of $P^-(G)$, where $G$ is simple and $|G|\rightarrow \infty$, was determined.

It is easy to see that the Guralnick--Kantor result fails for invariable generation: $\Lambda_e(G)$ can have isolated vertices (for instance, a $3$-cycle in $A_n$ with $n$ even). It seems to us that Theorems \ref{main_theorem_bounded_away_zero}, \ref{main_theorem_strengthening_bounded_rank} and \ref{main_theorem_absolute_constant_2} are essentially the best one can hope for in the invariable setting. Moreover, our results are among the first probabilistic statements concerning invariable generation of finite simple groups by \textit{two} elements. We are aware only of Shalev \cite[Theorem 4.2]{Sha}, which is Theorem \ref{main_theorem_absolute_constant_2} in case $G=\PSL_n(q)$ and $n$ large. 

Corollary \ref{main_theorem_absolute_constant} can be clearly stated in terms of $\Lambda_e(G)$ as follows.

\begin{corollary}
\label{c_graph}
Let $G$ be a nonabelian finite simple group with $|G|\to \infty$.
\begin{itemize}
    \item[(i)] Assume $G\not\cong G_2(q)$ when $3|q$, and assume that if $G=\PSp_{2m}(q)$ with $q$ even or $G=\POm_{2m+1}(q)$ with $q$ odd then $q\to \infty$. Then, the proportion of isolated vertices of $\Lambda_e(G)$ tends to zero as $|G|\to \infty$. More precisely, if we remove an asymptotically  negligible proportion of vertices from $\Lambda_e(G)$, and if we remove some further edges, we obtain a graph which is the union of at most $6$ stars.
    
    \item[(ii)] Assume $G$ is as in \eqref{eq_cane}, and assume that if $G=\PSp_{2m}(q)$ with $q$ even or $G=\POm_{2m+1}(q)$ with $q$ odd then $q$ is bounded. Then, the proportion of isolated vertices of $\Lambda_e(G)$ remains bounded away from zero as $|G|\to \infty$.
\end{itemize}
\end{corollary}

Corollary \ref{c_graph}(ii) reveals a surprising and sharp contrast with respect to the case of classical generation. Theorem \ref{main_theorem_absolute_constant_2}  and Corollary \ref{main_theorem_absolute_constant} can be seen also as a sort of ``invariable'' version of a concept introduced recently by Burness and Harper. In \cite{BH}, the \textit{total domination number} of a finite simple group $G$ is defined as the total domination number of $\Gamma(G)$, i.e., the minimal size of a subset $A$ of $G$ such that, if $1 \neq y \in G$, there exists $x \in A$ such that $\gen{x,y}=G$. Corollary \ref{main_theorem_absolute_constant} can be thought of as an analogue for invariable generation -- although again it is necessary to ignore a small proportion of elements.

\subsection{Methods} Here we briefly outline some of
the methods we use in the proofs of our main results.

Let $G$ be a finite group, and let $x \in G$. We define $\mathscr M(x)$ as the union of all conjugates of maximal subgroups of $G$ containing $x$. 
Equivalently, $\mathscr M(x)$ coincides with the union of all conjugacy classes of elements intersecting some maximal overgroup of $x$.

\begin{lemma}
\label{equivalence_inv_gen_tildas}
Let $A$ be a subset of $G$. Then,
\[
    1-\P_{\emph{inv}}(G,A)=\frac{|\bigcap_{x \in A} \mathscr M(x)|}{|G|}.
\]
\end{lemma}

\begin{proof}
Given $y \in G$, $\{x,y\}$ invariably generates $G$ if and only if $y \notin \mathscr M(x)$. The statement follows.
\end{proof}

Therefore our business is to find elements $x$ such that $\mathscr M(x)$ is small. This depends on two facts: 
\begin{itemize}

    \item[(I)] existence of elements lying in few maximal subgroups, and
    
    \item[(II)] existence of maximal subgroups $M$ of $G$ such that $\bigcup_{g\in G}M^g$ is small.
    
\end{itemize}
We must note that, in fact, (I) and (II) perform only part of the job. Indeed, taking the intersection of the sets $\mathscr M(x)$ is rather more delicate, and will require much more work. Moreover, of course the proof of the upper bound on $\P_{\text{inv}}(G,G)$ in Theorem \ref{main_theorem_absolute_constant_2} goes in the opposite direction.

Item (I) is a well studied topic (see, e.g.,  \cite{GK, Wei, BH}). Often, for us, applying results from these papers will be convenient, rather than essential. One reason is that, 
in our probabilistic approach, we can ignore the overgroups which are small (e.g., certain almost simple subgroups in groups of Lie type). 

For what concerns item (II), we will make essential use of deep results. We will use results of {\L}uczak--Pyber \cite{LP} for alternating groups (subsequently improved in Eberhard--Ford--Green \cite{EFG2} and Eberhard--Ford--Koukoulopoulos \cite{EFK}), and results of Fulman--Guralnick \cite{FG, FG3, FG2, FG4} for groups of Lie type. In these four papers, Fulman and Guralnick completed the proof of the so-called \textit{Boston--Shalev conjecture}, which asserts that for every finite simple group $G$ and for every proper subgroup $M$ of $G$, the proportion of elements belonging to $\cup_{g\in G}M^g$ is bounded away from $1$ absolutely (equivalently, the proportion of derangements in the action of $G$ on the cosets of $M$ is bounded away from zero absolutely). They proved  much stronger asymptotic results in many cases. 

One of our main tools 
is the intimate connection between the properties of invariable generation of a group of Lie type and the structure of its Weyl group. We will make this precise in Section \ref{section_bounded rank}. In bounded rank, this will allow us to translate the main theorems in terms of maximal tori (see e.g. Theorem \ref{theorem_1.5_restatement}). We will exploit the connection also in large rank, where the asymptotic properties of the Weyl groups will be relevant.


We remark again that it is enough to prove Theorem \ref{main_theorem_bounded_away_zero} for sufficiently large finite simple groups, since in \cite{GM} and \cite{KLS} it was proved that every finite simple group is invariably generated by two elements. What is more, for groups of Lie type we can divide the proof of Theorem \ref{main_theorem_bounded_away_zero} in two steps: first produce an element $x_1$ for groups of sufficiently large rank, and then produce an element $x_2$ for groups of bounded rank and sufficiently large fields. 

Finally, we note that we are free to define the subsets $A_b$ and $A_\ell$ from Theorem \ref{main_theorem_absolute_constant_2} only for sufficiently large finite simple groups. In fact, in groups of Lie type, for the set $A_\ell$ we may assume that $r$ is sufficiently large, and for the set $A_b$ we may assume (say) that $q \geq Cr^r$ for some large constant $C$. 
Clearly, the proof of Theorem \ref{main_theorem_absolute_constant_2} also splits naturally into the bounded rank
and large rank cases.



\section{Alternating groups}
\label{section_alternating_groups}
In this section we prove Theorems \ref{main_theorem_bounded_away_zero} and \ref{main_theorem_absolute_constant_2} for alternating groups. Conceptually, the proof follows from \cite{LP}. We will make use of \cite{EFG2} and \cite{EFK} in order to obtain better bounds.

\begin{theorem}\label{alternating_theorems_1_5}
The conclusions to Theorems~\ref{main_theorem_bounded_away_zero} and \ref{main_theorem_absolute_constant_2} hold in the case $G=A_n$.
\end{theorem}
\begin{proof}
If $n$ is odd, choose $x\in G=A_n$ to be an $n$-cycle. If $n$ is even, choose $x\in G$ to have cycle type $(n/2, n/2)$. Then set $A_\ell=\{x\}$. In the first case, the overgroups of $x$ are transitive subgroups, while in the second case the overgroups of $x$ are either transitive, or fix a set of size $n/2$.

By \cite[Theorem 1.1]{EFK}, the proportion of elements of $A_n$ lying in proper transitive subgroups of $A_n$ is $O(n^{-0.08})$. By \cite[Theorem 1.1]{EFG2}, the same bound holds for the proportion of elements fixing a set of size $n/2$. Therefore
\[
\frac{|\mathscr M(x)|}{|G|}=O(n^{-0.08}).
\]
Then $\P_{\text{inv}}(G,A_\ell)=1-O(n^{-0.08})$ by Lemma \ref{equivalence_inv_gen_tildas}. (We mention that, in \cite{EFK} and \cite{EFG2}, much more precise estimates are proved than those used here.)

We have proved Theorem~\ref{main_theorem_absolute_constant_2} for alternating groups, and since $|A_\ell|=1$, this implies Theorem \ref{main_theorem_bounded_away_zero} in this case.
\end{proof}

\section{Groups of Lie type of bounded rank: preliminaries}
\label{section_bounded rank}

In this section we introduce all the machinery that will lead us to the proof of the main theorems for groups of Lie type of bounded rank. We will prove Theorem \ref{main_theorem_absolute_constant_2}(1), as well as Theorem \ref{main_theorem_absolute_constant_2}(3) for $G=G_2(3^a)$, in Sections \ref{subsection_exceptional_groups} and \ref{subsection_classical_groups}. We will deduce Theorems \ref{main_theorem_strengthening_bounded_rank} and \ref{main_theorem_two_random_elements_bounded_rank} in Section \ref{section_proof_theorem_1.2}. 

We single out a special case.

\begin{theorem}
\label{main_psl_2}
Theorems \ref{main_theorem_strengthening_bounded_rank}, \ref{main_theorem_two_random_elements_bounded_rank}, \ref{main_theorem_absolute_constant_2} hold in case $G\cong \emph{PSL}_2(q)$.
\end{theorem}

The subgroup structure of $\PSL_2(q)$ is very well known and it is easy to prove Theorem \ref{main_psl_2}. We will do this at the beginning of Section \ref{subsection_classical_groups}.

The reason why we separate out this case is minor. Indeed, we will give an argument which works in general, but which gives error terms in $q$ of type $O(1/q^{1/2})$ if $G\cong \PSL_2(q)$, and of type $O(1/q)$ otherwise (see Remark \ref{remark_value_a}). We then prefer to consider $\PSL_2(q)$ separately, and deal with the other cases uniformly.

 Let $X$ be a simple linear algebraic group (of any isogeny type) over an algebraic closure $k$ of a finite field of characteristic $p$ (for all this theory, our reference is \cite{MT}). Let $\sigma$ be an endomorphism of $X$ such that the set $X_{\sigma}$ of fixed points of $\sigma$ is a finite group, and such that the derived subgroup $[X_\sigma, X_\sigma]=(X_\sigma)'$ is a perfect group. Let $T$ be a $\sigma$-stable maximal torus of $X$. Then $\sigma$ acts naturally on the character group $\text{Hom}(T,\GL_1)$. It turns out that the eigenvalues of $\sigma$ on $\text{Hom}(T,\GL_1)\otimes_{\mathbf Z} \mathbf C$ have all the same absolute value, which we denote by $q$, and which is a fractional power of $p$ (see \cite[Lemma 22.1 and Proposition 22.2]{MT}). We will write $X_\sigma=X_q$.


\begin{remark}
\label{rem:xq}
    The notation $X_q$ does not determine uniquely a finite group, since $q$ does not determine uniquely $\sigma$. For instance, both the groups $\SL_n(q)$ and $\SU_n(q)$ can be written as $X_q$ with $X=\SL_n$. 

    Recall, moreover, that $q$ is an integer, except for Suzuki and Ree groups, where $q$ is not an integer but $q^2$ is an integer (cf. Remark \ref{r_notation}(iii)).
\end{remark}


In Sections \ref{section_bounded rank}--
\ref{section_lower_bound_|A|}, we fix $X$, and we let $\sigma$ vary -- concretely, for classical groups we are fixing the rank and we are letting $q$ go to infinity, and moreover we are dealing with the exceptional groups.


\subsection{Subgroups of maximal rank}
\label{subsection_maximal_rank}
We begin by recalling a well known fact.

\begin{lemma}
\label{semisimple_in_stable_torus}
Assume $H\leq X$ is closed, connected and $\sigma$-stable, and assume $s\in H_\sigma$ is semisimple. Then, $s$ belongs to a $\sigma$-stable maximal torus of $H$.
\end{lemma}

\begin{proof}
We have that $s$ is contained in a maximal torus $S$ of $H$. Then $s\in S\leq \C_H(s)^\circ$. Since $s$ is central in $\C_H(s)^\circ$, $s$ is contained in every maximal torus of $\C_H(s)^\circ$ (this is also a maximal torus of $H$).
Now $\C_H(s)^\circ$ is $\sigma$-stable and connected, hence by Lang--Steinberg it contains a $\sigma$-stable maximal torus $L$ (see \cite[Theorem 21.11]{MT}). 
We have $s\in L$ and we are done.
\end{proof}

A proper closed subgroup $K$ of $X$ is said to be \textit{of maximal rank} if it contains a maximal torus of $X$. A subgroup of $X_{\sigma}$ is said to be \textit{of maximal rank} if it is of the form $K_{\sigma}$, where $K$ is a $\sigma$-stable subgroup of maximal rank (by Lang--Steinberg, $K$ contains a $\sigma$-stable maximal torus).

We list some notation that we will keep throughout Sections \ref{section_bounded rank}--\ref{section_lower_bound_|A|}. 
We advise the reader to consult this list whenever they find an unknown symbol, rather than to read all the items now. We prefer to amass this notation here, since we will use it several times in several different places.

\begin{notation}\label{notation}

\begin{itemize}
\item[(i)] We write $X_\sigma=X_q$, where ``$q$'' was defined in the paragraph preceding Remark \ref{rem:xq}.
    \item[(ii)] For a subgroup $H$ of $X_\sigma$, set
    \[
    \wt H=\bigcup_{g\in X_\sigma} H^g.
    \]
    \item[(iii)] $p$ denotes the characteristic of the field $k$.
    \item[(iv)] $r$ denotes the rank of $X$ (i.e, the dimension of a maximal torus). By Theorem \ref{main_psl_2}, we may assume $r\geq 2$, but we will make the requirement explicit.
    \item[(v)] $\ca M=\ca M(X_{\sigma})$ denotes the set of maximal subgroups of $X_{\sigma}$ of the form $K_{\sigma}$, where $K$ is a maximal $\sigma$-stable subgroup of $X$ of maximal rank.
    \item[(vi)] $\mathcal M_{\text{con}}=\mathcal M_{\text{con}}(X_{\sigma})$ denotes the set of subgroups of $X_{\sigma}$ of the form $(K^{\circ})_{\sigma}$, where $K$ is a maximal $\sigma$-stable subgroup of $X$ of maximal rank and $K^\circ$ denotes its connected component.
    \item[(vii)] For $x \in X_{\sigma}$, $\ca M(x)$ denotes the set of all conjugates of overgroups of $x$ belonging to $\ca M$.
    \item[(viii)] For $x\in X_\sigma$, $\ca T(x)$ denotes the set of maximal tori of $X_\sigma$ contained in some $K^\circ_\sigma$, where $K_\sigma \in \ca M(x)$ (note that $K^\circ_\sigma$ need not contain $x$).
    \item[(ix)] For a subset $A\subseteq X_\sigma$, $\P^*_{\text{inv}}(X_\sigma, A)$ denotes the probability that, if $y \in X_\sigma$ is chosen uniformly at random, there exists $x \in A$ such that for every $g_1, g_2 \in X_\sigma$, every maximal overgroup of $\gen{x^{g_1}, y^{g_2}}$ in $X_\sigma$ contains $(X_\sigma)'$. (In particular $\P^*_{\text{inv}}(X_\sigma, A)= \P_{\text{inv}}(X_\sigma, A)$ if $X_\sigma$ is perfect.)
\item[(x)] $\Delta=\Delta(X_\sigma)$ denotes the set of elements $y$ of $X_\sigma$ which are regular semisimple, and such that if $y$ belongs to a maximal subgroup $M$ of $X_\sigma$, then either $(X_\sigma)' \leq M$, or $M=K_\sigma \in \ca M$ and $y\in K^\circ_\sigma$.
\item[(xi)] For a maximal torus $S$ of $X_\sigma$, set $\Delta_S=\wt S \cap \Delta$.
\end{itemize}
\end{notation}
We recall a theorem which is essential for our purposes. Note that the proportion in the statement is independent of $X$, hence the result can be applied to groups of growing Lie rank (indeed we will use it in Section \ref{section_large_rank}).

\begin{theorem} \cite[Theorem 1.1]{GL}
\label{almost_all_elements_regular_semisimple}
The proportion of regular semisimple elements of $X_q$ is $1-O(1/q)$. 
\end{theorem}

For Suzuki and Ree groups the proportion is in fact $1-O(1/q^2)$, but we will not use this. The proof of the following theorem is largely contained in \cite{FG}, although at various places we need more explicit bounds.

\begin{theorem}
\label{almost_all_connected_components_maximal_rank}
Assume $r\geq 2$. We have
\[
\frac{|\Delta|}{|X_q|}=1-\frac{O(r^r)}{q}.
\]
\end{theorem}

\begin{proof}
Clearly we can assume $r\leq q$, otherwise the statement is empty. 
Let $A_1$ be the set of elements in $X_q$ which are not regular semisimple. Let $A_2$ be the set of elements which belong to maximal subgroups of $X_q$ which do not contain any maximal torus and which do not contain $(X_q)'$. Let $A_3$ be the set of elements which belong to $K_\sigma \setminus K^\circ$ for some maximal $\sigma$-stable subgroup $K$ of $X$ of maximal rank. We need to prove $|A_1\cup A_2\cup A_3|/|X_q| =O(r^r/q)$.

We have $|A_1|/|X_q|=O(1/q)$ by Theorem \ref{almost_all_elements_regular_semisimple}. 

Next we deal with $A_2$. Let $\Omega$ be the set of maximal subgroups of $X_q$ which do not contain $(X_q)'$, which do not contain any maximal torus of $X_q$, and which are not subfield subgroups (cf. \cite[Section 3]{FG}). If $M\in \Omega$ then $M$ has $O(q^{r-1})$ conjugacy classes (see \cite{FG} and the proof of \cite[Theorem 7.3]{FG3}). 

Assume now $x\in X_q$ is regular semisimple. Then the $X_q$-class of $x$ has size $O(|X_q|)/(q-1)^r$. 
By Theorem \ref{almost_all_elements_regular_semisimple}, we see that if $M\in \Omega$ then
\[
\frac{|\wt M|}{|X_q|} = \frac{O(q^{r-1})}{(q-1)^r} +\frac{O(1)}{q} = \frac{O(1)}{q},
\]
where in the last equality we used $r\leq q$ (in fact $r=O(q)$ is enough). It is known (cf. \cite[Theorem 1.3]{LMSh}) that the number of conjugacy classes of subgroups in $\Omega$ is bounded by a function of $r$. (Note that $X_q$ surjects, with central kernel, onto an almost simple group generated by inner-diagonal automorphisms.) 
In case $X$ is classical, by \cite[Theorem 1.2]{GLT} we can take this function to be $O(r^6)$.

Next we deal with subfield subgroups. The argument given in \cite[Lemma 3.7]{FG} shows that the proportion of elements belonging to subfield subgroups is $O(r/q^{r/2})+O(1/q)$, which is $O(r/q)$ since $r\geq 2$. Therefore $|A_2|/|X_q|=O(r^6/q)$. 

Finally we deal with $A_3$. Let $K$ be a maximal $\sigma$-stable closed subgroup of $X$ of maximal rank. We claim that
\[    
\frac{|\bigcup_{g\in X_q}(K_\sigma \setminus K^\circ)^g|}{|X_q|}=\frac{O(|K_\sigma : K^\circ_\sigma|)}{q}. 
\]
We go through the proof of \cite[Proposition 4.2]{FG}: we use the same arguments and the same computations, except that we bound the size of a regular semisimple class by $O(|X_q|)/(q-1)^r$, and moreover we use $r=O(q)$, so that $((q+1)/(q-1))^{r-1}$ is bounded.

At this point we can deduce that $|A_3|/|X_q|=O(r^r/q)$. Indeed, it is known that $|K_\sigma : K^\circ_\sigma|=O((r+1)!)$, and moreover the number of $X_q$-conjugacy classes of maximal subgroups of maximal rank is linear in $r$, from which $|A_3|/|X_q|=O(r^r/q)$. (These facts are known in a very precise way, cf. \cite{LSS} and \cite{LSe}. We will recall them in Sections \ref{subsection_exceptional_groups} and \ref{subsection_classical_groups}. The term $(r+1)!$ can occur for stabilizers of direct sum decompositions in classical groups.)

Putting together the bounds given for $A_1$, $A_2$ and $A_3$, we get the result.
\end{proof}

\begin{remark}
\label{remark_value_a}
By the same argument as in the previous proof, the proportion of elements of $\SL_2(q)$ belonging to subfield subgroups is $O(1/q^{1/2})$. If $q$ is a square, the proportion of elements inside a conjugate of $\SL_2(q^{1/2})$ is indeed of this form. This is the only reason for which we have considered this case separately.
    
We also note that, in bounded rank, an essential part of our method is to focus on regular semisimple elements. By work of Guralnick--L\"{u}beck \cite{GL} and Fulman--Neumann--Praeger \cite{FNP}, it is known that, except for Suzuki and Ree groups, the proportion of elements of $X_q$ which are not regular semisimple is comparable to $1/q$ (up to constants). Therefore, with our method we cannot get error terms in $q$ which are better than $O(1/q)$.
\end{remark}

\subsection{Maximal tori and Weyl group}
\label{subsection_weyl_group_tori}
There is a well known and fundamental connection between the maximal tori of $X_\sigma$ and the Weyl group of $X$, which now we recall (see \cite[Section 25]{MT} for the general theory). Together with Theorems \ref{almost_all_elements_regular_semisimple} and \ref{almost_all_connected_components_maximal_rank}, this will enable us to translate the main theorems in terms of maximal tori. 

Throughout this subsection, we fix a $\sigma$-stable maximal torus $T$, and let $W=\N_X(T)/T$ be the Weyl group of $X$ with respect to $T$. Then $\sigma$ acts on $W$. There is a bijection between $X_{\sigma}$-conjugacy classes of $\sigma$-stable maximal tori of $X$ and $W$-conjugacy classes contained in the coset $\sigma W$ of the group $\gen{\sigma}\ltimes W$ (if $\sigma$ acts trivially on $W$, these can be identified with the conjugacy classes of $W$). If $w \in W$, we denote by $T_w$ any representative of the conjugacy class of maximal tori corresponding to the $W$-class of $\sigma w$. We have $T_w=T^g$, where $g\in X$ is such that $g^\sigma g^{-1}$ maps to $w\in W$. Moreover $\N_{X_{\sigma}}(T_w)/(T_w)_{\sigma} \cong \C_W(\sigma w)$. 

Let $\Psi\subseteq W$ be such that $\{\sigma w : w\in \Psi\}$ is a set of representatives for the $W$-classes in the coset $\sigma W$. Let $\Omega$ be a subset of $\Psi$. 
Denote by $\P(W,\sigma, \Omega)$ the probability that a random element of $\sigma W$ is $W$-conjugate to $\sigma w$ for some $w\in \Omega$. In case $\Omega=\{w\}$, we will write $\P(W,\sigma, w)$ instead of $\P(W,\sigma, \{w\})$. Using that $\N_{X_{\sigma}}(T_w)\leq \N_{X_{\sigma}}((T_w)_{\sigma})$, 
by a trivial union bound we get
\begin{equation}
\label{equation_inequality_tori}
    \frac{|\bigcup_{w \in \Omega}\wt{(T_w)_{\sigma}}|}{|X_{\sigma}|}\leq \sum_{w \in \Omega}\frac{1}{|\C_W(\sigma w)|}=\P(W,\sigma, \Omega).
\end{equation}
Despite being trivial, for $q$ large this bound is accurate.

\begin{theorem}
\label{correspondence_tori_weyl_group}
\begin{equation}
\label{equation_maximal_tori_classes}
    \frac{|\bigcup_{w \in \Omega}\wt{(T_w)_{\sigma}}|}{|X_q|}=\P(W,\sigma, \Omega)+\frac{O(1)}{q}.
\end{equation}
\end{theorem}

\begin{proof}
By Lemma \ref{semisimple_in_stable_torus} and Theorem \ref{almost_all_elements_regular_semisimple} we have 
\[
\beta:=\frac{|\bigcup_{w \in \Omega}\wt{(T_w)_{\sigma}}|}{|X_q|}+\frac{|\bigcup_{w \in \Psi\setminus\Omega}\wt{(T_w)_{\sigma}}|}{|X_q|} = 1-\frac{O(1)}{q}.
\]
(Note that $\beta\leq 1$ by \eqref{equation_inequality_tori}.) Moreover
\begin{align*}
    \frac{O(1)}{q} =1-\beta &=\left(\P(W,\sigma,\Omega) - \frac{|\bigcup_{w \in \Omega}\wt{(T_w)_{\sigma}}|}{|X_q|}\right)\\ &+ \left(\P(W,\sigma,\Psi\setminus \Omega) - \frac{|\bigcup_{w \in \Psi\setminus \Omega}\wt{(T_w)_{\sigma}}|}{|X_q|}\right),
\end{align*}
where by \eqref{equation_inequality_tori} both summands are nonnegative. In particular they are both $O(1/q)$, and the proof is concluded.
\end{proof}

\begin{remark}
In fact, the previous proof shows that
\[
\P(W,\sigma,\Omega) -  \frac{|\bigcup_{w \in \Omega}\wt{(T_w)_{\sigma}}|}{|X_q|} 
\]
is at most the proportion of elements of $X_q$ which are not semisimple. Therefore, we are really using that the proportion of \textit{semisimple} elements of $X_q$ is $1-O(1/q)$, rather than Theorem \ref{almost_all_elements_regular_semisimple}.
\end{remark}

Theorem \ref{correspondence_tori_weyl_group} was used in \cite{FG, FG2} (although we have not found the above proof in the literature). Note, once again, that the error term is independent of $r$, hence the result can be applied to groups of growing Lie rank. 

In this section, the key assumption is that $X$ is fixed. Then $W$ is fixed, and if $\Omega$ is nonempty the expression on the right-hand side of \eqref{equation_maximal_tori_classes} is always bounded away from zero: it is at least $1/|W|+O(1/q)$.

For a subset $A$ of $X_q$, let $T_1, \ldots, T_\ell$ be a set of representatives of the $X_q$-conjugacy classes of members of $\cap_{x\in A}\ca T(x)$ (possibly $\ell=0$). Write $T_i=(T_{w_i})_{\sigma}$, where $T_{w_i}$ is a $\sigma$-stable maximal torus of $X$ and $w_i\in W=\N_X(T)/T$. Set $\Omega=\{w_1, \ldots, w_\ell\}$. 

\begin{theorem}
\label{only_maximal_tori_matter}
Assume $r\geq 2$. We have
\begin{equation}
\label{equation_only_maximal_tori_matter}
1-\P^*_{\emph{inv}}(X_q,A) = \P(W,\sigma,\Omega)+\frac{O(r^r)}{q}.
\end{equation}
\end{theorem}

\begin{proof}
Reasoning as in Lemma \ref{equivalence_inv_gen_tildas}, and using Theorem \ref{almost_all_connected_components_maximal_rank}, we have
\[
    1-\P^*_{\text{inv}}(X_q,A)=\frac{|\bigcap_{x \in A} \bigcup_{M\in \ca M(x)}M|}{|X_q|} + \frac{O(r^r)}{q}.
    \]
Now we look at the right-hand side of the above equation.
Assume $y\in \Delta$. Then $y$ is regular semisimple; let $S=\C_X(y)^{\circ}$ be its maximal torus in $X$. Assume $y\in K_\sigma$ for some $K_\sigma \in \ca M(x)$ and some $x \in A$ (and $K$ is $\sigma$-stable of maximal rank). By definition of $\Delta$ we have $y\in K^\circ$. By Lemma \ref{semisimple_in_stable_torus}, $y$ lies in some $\sigma$-stable maximal torus of $K^{\circ}$, which is also a maximal torus of $X$, hence must coincide with $S$. In particular, if $y$ lies in some member of $\ca M(x)$ for every $x \in A$, then $S$ belongs to $\ca T(x)$ for every $x \in A$. Using Theorem \ref{almost_all_connected_components_maximal_rank}, this shows that
\begin{equation}
\label{equation_proof}
1-\P^*_{\text{inv}}(X_q,A) = \frac{|\bigcup_{i=1}^\ell \wt{(T_{w_i})_{\sigma}}|}{|X_q|}+\frac{O(r^r)}{q}.
\end{equation}
Finally, the right-hand side of \eqref{equation_proof} is equal to the right-hand side of \eqref{equation_only_maximal_tori_matter} by Theorem \ref{correspondence_tori_weyl_group}. 
\end{proof}

We record a consequence of the previous proof. 

\begin{remark}
\label{theorem_consequence_proof}
Assume $\cap_{x\in A}\ca T(x)=\emptyset$. Then the set $\Delta$ contributes to \\ $\P^*_{\text{inv}}(X_q,A)$. In other words, for every $y\in \Delta$, there exists $x\in A$ such that, for every $g_1, g_2 \in X_q$, every maximal overgroup of $\gen{x^{g_1}, y^{g_2}}$ in $X_\sigma$ contains $(X_\sigma)'$.
\end{remark}

\subsection{From $X_\sigma$ to $(X_\sigma)'$}
\label{subsection_isogeny}
All the discussion above is about $X_\sigma$, which need not be perfect. 
 Here we establish the connection to the
corresponding finite simple groups, which we need in order to prove our main results. In addition,
we will show that the isogeny type of $X$ is not relevant. Let $X_{\text{sc}}$ be the group of simply connected type, and let $\pi:X_{\text{sc}}\rightarrow X$ be the natural isogeny. Then $\sigma$ lifts to a morphism $X_{\text{sc}}\rightarrow X_{\text{sc}}$ (see \cite[Proposition 22.7]{MT}), which for convenience we still denote by $\sigma$. Write as usual $X_\sigma=X_q$ and $(X_{\text{sc}})_\sigma = (X_{\text{sc}})_q$.

\begin{lemma}
\label{isogeny_types}
Let $A$ be a subset of $(X_{\emph{sc}})_q$. Then
\[
    \P_{\emph{inv}}((X_q)',A^\pi)=\P_{\emph{inv}}((X_{\emph{sc}})_q,A)=\P^*_{\emph{inv}}(X_q,A^\pi)+\frac{O(r^r)}{q}.
\]
\end{lemma}

\begin{proof}
Let $Z$ be the kernel of $\pi$. Then $(X_{\text{sc}})_\sigma/Z_\sigma \cong ((X_{\text{sc}})_\sigma)^\pi = (X_\sigma)'$ (\cite[Proposition 24.21]{MT}). Moreover $Z_\sigma$ is contained in every maximal subgroup of $(X_{\text{sc}})_\sigma$, hence 
the first equality of the statement holds.

Now note that $\P_{\text{inv}}((X_{\text{sc}})_q,A)=\P^*_{\text{inv}}((X_{\text{sc}})_q,A)$, since $(X_{\text{sc}})_q$ is perfect. Since $Z$ is contained in every maximal torus of $X_{\text{sc}}$, 
$\pi$ induces a bijection between $\sigma$-stable subgroups of maximal rank of $X_{\text{sc}}$ and of $X$, which maps overgroups of $y\in A$ to overgroups of $y^\pi \in A^\pi$. Then the second equality follows from Theorem \ref{only_maximal_tori_matter}. 
\end{proof}
In order to prove Theorem \ref{main_theorem_absolute_constant_2}(1), in view of Theorem \ref{only_maximal_tori_matter} and Lemma \ref{isogeny_types} it is sufficient to choose $X$ of some isogeny type, and prove the following statement.

\begin{theorem}
\label{theorem_1.5_restatement}
Assume $r\geq 2$ and $q\geq Cr^r$ for some large constant $C$, and assume $X_q\not\cong G_2(q)$ when $3\mid q$. Then, there exists $A_b\subseteq (X_q)'$ of size as in Table \ref{main_table} such that $\cap_{x\in A_b}\ca T(x)=\emptyset$.
\end{theorem}

We will prove Theorem \ref{theorem_1.5_restatement} in Sections \ref{subsection_exceptional_groups} and \ref{subsection_classical_groups}.

\subsection{Reductive subgroups of maximal rank}
\label{subsection_reductive_max_rank}
By the Borel--Tits theorem (see \cite[Corollaire 3.9]{BT}), a maximal $\sigma$-stable subgroup of $X$ of maximal rank is either parabolic, or its connected component is reductive. In this subsection we make some general observations regarding the second case.

Let $T$ be a $\sigma$-stable maximal torus of $X$, let $\Phi$ be the root system with respect to $T$, and denote by $U_\alpha$, $\alpha \in \Phi$, the corresponding root subgroups. Let $W=\N_X(T)/T$ be the Weyl group of $X$ with respect to $T$. 


Assume that $\sigma$ acts trivially on $\Phi$. This is not essential, but it makes the statement of Lemma \ref{bijection_conjugacy_classes_W_Psi}, below, easier (and, of course, we will need the lemma only under this assumption).

The following discussion is taken from \cite{LSS}. Let $K$ be a closed connected reductive subgroup of $X$ containing $T$. Then, $K=\gen{T, U_\alpha, \alpha \in \Psi}$ for a $p$-closed subset $\Psi$ of $\Phi$ (see \cite[Section 13]{MT} for this notion). Let $W(\Psi)$ be the Weyl group of $K$, i.e., the subgroup of $W$ generated by the reflections in roots of $\Psi$. We have $\N_X(K)/K\cong \N_W(W(\Psi))/W(\Psi) =: W_\Psi$. Note that $K$ is $\sigma$-stable, since $\sigma$ acts trivially on $\Phi$.

Assume now $H$ is a $\sigma$-stable conjugate of $K$. In particular, there exists $g\in X$ such that $H=K^g$ and $T^g$ is $\sigma$-stable. 
Then $g^\sigma g^{-1}\in \N_X(T)\cap \N_X(K)$, which maps to an element of $W_\Psi$ that we denote by $\rho(K^g)$. 


\begin{lemma}
\label{bijection_conjugacy_classes_W_Psi}
Assume that $\sigma$ acts trivially on $\Phi$. The map $\rho$ defined above induces a well-defined bijection between $\{X_\sigma$-orbits on the $\sigma$-stable conjugates of $K\}$ and \{conjugacy classes in $W_\Psi\}$.
\end{lemma}
\begin{proof}
This is \cite[Propositions 1 and 2]{Car2} in case $\sigma$ acts trivially on $\Phi$.
\end{proof}
In \cite[Propositions 1 and 2]{Car2} the general case, in which $\sigma$ does not necessarily act trivially, is considered. This is more technical to state.

Next, given a $\sigma$-stable maximal torus $S$, we want to determine its closed connected reductive overgroups in $X$. We do not need to assume that $\sigma$ acts trivially, since the argument is the same. Fix $g$ such that $S=T^g$ and let $w$ be the image of $g^\sigma g^{-1}$ in $W$.

The following result appears as \cite[Theorem 5]{Wei}, although the language and the proof seem slightly different.

\begin{lemma}
\label{reductive_overgroups_of_fixed_torus}
The closed connected $\sigma$-stable reductive overgroups of $S$ are in bijection with $p$-closed subsets of $\Phi$ which are $\sigma w$-stable.
\end{lemma}
\begin{proof}
If $H$ is a closed connected $\sigma$-stable reductive overgroup of $S$, set $K:=H^{g^{-1}}$. Then $K=\gen{T, U_\alpha, \alpha \in \Psi}$ for some (unique) $p$-closed subset $\Psi$ of $\Phi$. Moreover, $\sigma g^\sigma g^{-1}$ normalizes $K$ and $T$, hence $\sigma w$ fixes $\Psi$. 
Conversely, assume $\Psi$ is $\sigma w$-stable and $p$-closed; then $\gen{T, U_\alpha, \alpha \in \Psi}^g$ is a closed connected reductive overgroup of $S$ (see \cite[Theorem 13.6]{MT}). For $\alpha \in \Psi$ we have $(U_\alpha^g)^{\sigma} = U_\alpha ^{\sigma g^\sigma}=(U_\alpha ^{\sigma g^\sigma g^{-1}})^g=(U_{\alpha \sigma w})^g$. Since $\Psi$ is $\sigma w$-stable, we deduce that $\gen{T, U_\alpha, \alpha \in \Psi}^g$ is $\sigma$-stable.
\end{proof}


Of course, if $\sigma$ acts trivially on $\Phi$, then a subset of $\Phi$ is $\sigma w$-stable if and only if it is $w$-stable.

At this point we divide the discussion between exceptional and classical groups.

\section{Exceptional groups}
\label{subsection_exceptional_groups}

In this section we will prove Theorem \ref{theorem_1.5_restatement} for simple exceptional groups, which implies Theorem \ref{main_theorem_absolute_constant_2}(1) for these groups, and we will prove Theorem \ref{main_theorem_absolute_constant_2}(3) for $G_2(3^a)$.

We choose $X$ of adjoint type. 
The maximal subgroups of $X_\sigma$ of maximal rank have been classified by Liebeck--Saxl--Seitz \cite{LSS}.

We can assume that $q$ is sufficiently large in the proof. This implies that every maximal torus $S_\sigma$ of $X_q$ contains regular semisimple elements. (In fact, more is true. By Theorem \ref{almost_all_elements_regular_semisimple}, almost all elements of $X_q$ are regular semisimple. By Theorem \ref{correspondence_tori_weyl_group}, the proportion of elements of $X_q$ lying in $\wt{S_\sigma}$ is bounded away from zero; therefore, almost all elements of $\wt{S_\sigma}$ are regular semisimple.) In particular, it follows that whenever $S_\sigma \leq M^\circ_\sigma$, with $M_\sigma \in \ca M$, then $S\leq M^\circ$. 


\begin{table}
\small
\centering
\caption{$A_b=\{x_1,x_2\}$ in Theorem \ref{main_theorem_absolute_constant_2} for exceptional groups $G\neq 
 F_4(2^a)$.}
\label{table_exceptional_groups}       
\begin{tabular}{lll}
\hline\noalign{\smallskip}
$G$ & $|x_1|$ & $|x_2|$\\
\noalign{\smallskip}\hline\noalign{\smallskip}
$^2\!B_2(q^2)$  & $\Phi'_8$ &  $\Phi'_8(-q)$ \\
$^2\!G_2(q^2)$  & $\Phi'_{12}$ & $\Phi'_{12}(-q)$ \\
$G_2(q)$, $3\nmid q$ &  $\Phi_3$ & $\Phi_3(-q)$  \\
$^3\!D_4(q)$ & $\Phi_{12}$  & $(q^3+1)(q-1)/(2,q-1)$ \\
$^2\!F_4(q^2)$ & $\Phi'_{24}$  & $\Phi'_{24}(-q)$ \\
$F_4(q)$, $q$ odd & $\Phi_{12}$  & $\Phi_8$ \\
$E_6(q)$ & $\Phi_3\Phi_{12}/(3,q-1)$ & $\Phi_1\Phi_2\Phi_8/\delta$ \\
$^2\!E_6(q)$ & $\Phi_6\Phi_{12}/(3,q+1)$ & $\Phi_1\Phi_2\Phi_8/\delta'$ \\
$E_7(q)$ & $\Phi_2 \Phi_{18}/(2,q-1)$ & $\Phi_1 \Phi_9/(2,q-1)$  \\
$E_8(q)$ & $\Phi_{30}$ & $\Phi_{30}(-q)$ \\
\noalign{\smallskip}\hline
\end{tabular}
\end{table}

We define $A_b$ as the set of elements of $(X_q)'$ appearing in Table \ref{table_exceptional_groups}. We assume that $X_q \neq G_2(2), ^2\!G_2(3), ^2\!F_4(2)$. In the table, the case $F_4(q)$ with $q$ even is missing. In this case the set $A_b$ has size $6$, hence for aesthetic reasons we have not included it. We will treat this case in detail in Subsection \ref{subsection_F4}.

Each element in Table \ref{table_exceptional_groups} is regular semisimple. The existence of these elements follows from the general theory of the structure of maximal tori, cf. \cite[Section 25]{MT}. 
For the element $x_2$ in case $G=E_6(q)$, $\delta=(3,q-1)(4,q-1)$. For the element $x_2$ in case $G=^2\hspace{-0.5em} E_6(q)$, $\delta'=(3,q+1)(4,q+1)$.

We write $\Phi_n=\Phi_n(q)$ for the $n$-th cyclotomic polynomial evaluated at $q$. Moreover $\Phi'_8=\Phi'_8(q)=q^2+\sqrt 2 q + 1$, $\Phi'_{12}=\Phi'_{12}(q)=q^2+\sqrt 3 q +1$, $\Phi'_{24}=\Phi'_{24}(q)=q^4+\sqrt 2 q^3+q^2+\sqrt 2 q +1$ (this notation is taken from \cite{GM2}). We will refer to \cite[Table 6]{GM2} and \cite[Table 1]{GM} for the overgroups of many elements in Table \ref{table_exceptional_groups}. We remark that the aforementioned tables from \cite{GM2} and \cite{GM} rely mostly on \cite{Wei}.

\subsection{Some twisted groups, and $E_8(q)$}
In many cases we can exploit a very convenient situation. Indeed, consider the groups $^2\!B_2(q^2), ^2\!G_2(q^2), ^3\!D_4(q), ^2\!F_4(q^2)$ and $E_8(q)$. Then by \cite[Table 6]{GM2} we see that the element $x_1$ lies only in one maximal subgroup of $X_q$, namely $\N_{X_q}(S_\sigma)$, where $S_\sigma$ is the unique maximal torus of $X_q$ containing $x_1$. Since $q$ is large, $S_\sigma$ contains regular semisimple elements, and in particular $\N_{X_q}(S_\sigma)=\N_{X_q}(S)$. (We note that, in fact, $\N_{X_q}(S_\sigma)=\N_{X_q}(S)$ holds under the weaker hypothesis that $S_\sigma$ is \textit{nondegenerate}; see \cite[Section 3.6]{Car} for this notion.)
The connected component of $\N_X(S)$ is $S$, since $S$ has finite index in its normalizer. By definition, we deduce that $\ca T(x)$ contains only the conjugates of $S_\sigma$.

Then, in order to prove Theorem \ref{theorem_1.5_restatement} in these cases, we just need to show that the element $x_2$ does not belong to any conjugate of $\N_{X_q}(S_\sigma)$. This is easily done by order considerations.

\subsection{$E_6(q)$ and $^2\!E_6(q)$}
\label{sub_e6}
We write $E_6(q)=^+\hspace{-0.5em}E_6(q)$ and $^2\!E_6(q)=^-\hspace{-0.5em}E_6(q)$. Consider $^\varepsilon \! E_6(q)$ with $\varepsilon \in \{+,-\}$. By \cite[Table 1]{GM}, $x_1$ is contained only in $(^3\!D_4(q)\times (q^2+\varepsilon q+1)).3$ (among the maximal subgroups of $X_\sigma$). The order of $x_2$ is $(q^4+1)(q^2-1)/\delta_\varepsilon$, where $\delta_\varepsilon = (3,q-\varepsilon 1)h_\varepsilon$ and $h_\varepsilon=(4,q-\varepsilon 1)$. For $\varepsilon=-$, $x_2$ is contained in a maximal subgroup $M=h_\varepsilon.(\POm_{10}^\varepsilon(q)\times (q-\varepsilon)/h_\varepsilon).h_\varepsilon$; and for $\varepsilon=+$, $x_2$ is contained in a parabolic subgroup with Levi complement of type $D_5$ (cf. \cite[Table 5.1]{LSS}). By order considerations we see that if $\varepsilon=+$ then $\ca M(x_2)$ contains only parabolics of type $D_5$, while if $\varepsilon=-$ then $\ca M(x_2)$ contains the conjugates of $M$, and parabolics with Levi complement of type $^2\!D_4$. 
Using the knowledge of maximal tori of $^\varepsilon \!E_6(q)$ (see \cite{DF}), we deduce by order considerations that $\ca T(x_1)\cap \ca T(x_2)=\emptyset$, which proves Theorem \ref{theorem_1.5_restatement} in these cases.
In Subsections \ref{subsection_G2_not_3}--
\ref{subsection_F4}, we employ the notation of Subsection \ref{subsection_reductive_max_rank}.

\subsection{$G_2(q)$ with $3\nmid q$}
\label{subsection_G2_not_3}
Let $G=G_2(q)$ with $q\geq 3$. We immediately recall some facts regarding maximal tori of $G_2(q)$ that we will use also in Subsection \ref{subsection_G2_power_3}. We have $W=W(G_2)\cong D_{12}$, hence by the general theory (see \cite[Section 25]{MT}) there are six $G_2(q)$-classes of maximal tori, with representatives $T_1, \ldots, T_6$, and with orders $q^2-1, q^2-1, (q-1)^2, (q+1)^2, q^2+q+1, q^2-q+1$, respectively. We assume $T_i=(T_{w_i})_\sigma$, where $w_1$ is a reflection in a short root, $w_2$ is a reflection in a long root, $w_3=1$, $w_4=-1$, $|w_5|=3$, $|w_6|=6$. For $i=1, \ldots, 6$, we will write $\Delta_i$ instead of $\Delta_{T_i}$.

By Theorem \ref{correspondence_tori_weyl_group}, the proportion of elements lying in $\wt{T_1}\cup \wt{T_2}$ is equal to $O(1/q)$ plus the proportion of noncentral involutions of $D_{12}$, which is $1/2$. Consequently, the proportion of elements lying in $\cup_{i=3}^6\wt{T_i}$ is $1/2+O(1/q)$. By Theorem \ref{almost_all_connected_components_maximal_rank}, the same estimates hold if we replace each $\wt{T_i}$ by $\Delta_i$. 

Assume now $3\nmid q$. By \cite[Table 5.1]{LSS} and by order considerations, $\ca M(x_1)$ contains only the conjugates of $\SL_3(q).2$, and $\ca M(x_2)$ contains only the conjugates of $\SU_3(q).2$. In order to prove Theorem \ref{theorem_1.5_restatement} in this case, we need to show that these two subgroups do not contain a common maximal torus (up to conjugacy). By order considerations, if there exists a common torus of $\SL_3(q).2$ and $\SU_3(q).2$, then it must be $T_1=(T_{w_1})_\sigma$ or $T_2=(T_{w_2})_\sigma$. For $i=1,2$, fix $g_i\in X=G_2$ such that $T_{w_i}=T^{g_i}$ (and $g_i^\sigma g_i^{-1}$ maps to $w_i\in W$). By Lemma \ref{reductive_overgroups_of_fixed_torus}, the closed connected reductive subgroups of $G_2$ containing $T_{w_i}$ are precisely the subgroups $K(\Psi)^{g_i}$, where $K(\Psi)=\gen{T, U_\alpha, \alpha \in \Psi}$ and $\Psi$ is $p$-closed and $w_i$-stable. Since $3\nmid q$, by \cite[Theorem 13.14]{MT} we deduce that every $p$-closed subset of $\Phi$ is closed; in particular there is only one $p$-closed subset $\Psi$ of type $A_2$: the set of all long roots. Note that $W_\Psi=\N_W(W(\Psi))/W(\Psi)\cong C_2$, hence by Lemma \ref{bijection_conjugacy_classes_W_Psi} there are two corresponding $G_2(q)$-classes. Now $w_2\in W(\Psi)$, while $w_1\not\in W(\Psi)$; then by Lemma \ref{bijection_conjugacy_classes_W_Psi} $K(\Psi)^{g_2}$ is $G_2(q)$-conjugate to $K(\Psi)$, and $K(\Psi)_\sigma \cong \SL_3(q)$, while $K(\Psi)^{g_1}_\sigma \cong \SU_3(q)$. This implies that $\SL_3(q).2$ and $\SU_3(q).2$ do not contain a common maximal torus (up to conjugacy), and Theorem \ref{theorem_1.5_restatement} follows in this case.

\subsection{$G_2(q)$ with $3\mid q$}
\label{subsection_G2_power_3} We keep the notation from the beginning of Subsection \ref{subsection_G2_not_3}. Let $G=G_2(q)$. We want to prove $\P_{\text{inv}}(G,G)=1/2+O(1/q)$, namely, Theorem \ref{main_theorem_absolute_constant_2}(3) in this case. 
We will prove the following more precise statement, which we will use in Section \ref{section_proof_theorem_1.2}. We note that $x\in G$ contributes to $\P_{\text{inv}}(G,G)$ if and only if $\P_{\text{inv}}(G,x)>0$.

\begin{theorem} Let $G=G_2(q)$ with $3\mid q$.
\label{theoremg2order}
\begin{enumerate}
    \item[(i)] $\P_{\emph{inv}}(G,x)\geq 1/6+O(1/q)$ for a proportion of elements $x\in G$ of the form $1/2+O(1/q)$.
    \item[(ii)] $\P_{\emph{inv}}(G,x)=0$ for a proportion of elements $x\in G$ of the form $1/2+O(1/q)$.
\end{enumerate}
\end{theorem}

\begin{proof}

There is an automorphism $\gamma$ of $G_2(q)$ which induces a graph automorphism of order two on the Dynkin diagram, exchanging long and short roots. The set $\Psi'$ of short roots is $3$-closed (cf. \cite[Proposition 13.15]{MT}). There are two conjugacy classes of subgroups $\SL_3(q).2$, with representatives $H_1$ and $H_2$, and two conjugacy classes of subgroups $\SU_3(q).2$, with representatives $K_1$ and $K_2$. 
We have $H_1^\gamma = H_2$ and $K_1^\gamma=K_2$. Moreover $\gamma$ exchanges the classes of $T_1$ and $T_2$. Up to changing indices, $T_i$ is contained in a conjugate of $H_i$ and $K_i$. What is more, the only overgroups of $T_5$ (resp. $T_6$) are conjugates of $H_1$ and $H_2$ (resp. conjugates of $K_1$ and $K_2$).

Therefore, by definition of $\Delta$, every element of $\Delta_3$ (resp. $\Delta_4$, resp. $\Delta_5$) invariably generates with every element of $\Delta_6$ (resp. $\Delta_5$, resp. $\Delta_6$).  
We observed that the proportion of elements belonging to $\cup_{j=3}^6\Delta_j$ is $1/2+O(1/q)$. Moreover, for $j\in \{5,6\}$, we have $|\Delta_j|/|G|= 1/6+O(1/q)$. Therefore item (i) is proved.

We move to (ii). We want to show that
\begin{itemize}
    \item[$(\star)$] for every $x\in G$, $x$ belongs to a maximal subgroup containing a conjugate of $T_1$, and to a maximal subgroup containing a conjugate of $T_2$.
\end{itemize}
This implies that all elements $y$ lying in $\wt{T_1}\cup \wt{T_2}$ are such that $\P_{\text{inv}}(G,y)=0$. We observed that these elements have proportion $1/2+O(1/q)$, hence in order to prove (ii) we only need to prove $(\star)$.

It is sufficient to focus on $i=1$, since the two tori are exchanged by an automorphism of $G$. Representatives of the conjugacy classes of maximal subgroups containing $T_1$ are the following:
\[
    \{P, H_1, K_1, C\}.
\]
Here $P$ is a parabolic subgroup with respect to the short root of a base, and $C\cong (\SL_2(q)\circ \SL_2(q)).2$ is the centralizer of an involution in $G_2(q)$ (cf. \cite[Theorem A]{Kle}; recall that $G$ contains a unique conjugacy class of involutions). 

Let $x\in G$. If $x$ is unipotent, then $x$ is contained in both conjugacy classes of maximal parabolic subgroups. Assume then $x=su$, with $1 \neq s$ semisimple, $u$ unipotent, and $[s,u]=1$. Then $x \in \C_G(s)<G$. If $\C_G(s)$ is a maximal torus of even order, it is contained in a conjugate of $C$. The remaining classes of maximal tori have representatives $T_5$ (contained in $H_1$) and $T_6$ (contained in $K_1$). Examining \cite[Table II, p. 41]{Kle}, we see that all other possibilities for $\C_G(s)$ contain a central involution, hence are contained in a conjugate of $C$. Then $(\star)$ is proved and we are done.
\end{proof}


\subsection{$E_7(q)$}
\label{subsection_E7}
By \cite[Table 6]{GM2} and \cite[Table 1]{GM} we see that $x_1$ is contained only in a maximal subgroup $^2\!E_6(q)_{\text{sc}}.D_{q+1}$ of $X_\sigma$, and $x_2$ is contained in two (conjugate) parabolics $P$ and $P'$ of type $E_6$, and in the normalizer of a common Levi complement $L$. Our aim is to show that $\mathcal T(x_1)\cap \mathcal T(x_2)=\emptyset$.

\begin{claim}
\label{claim}
Assume $g\in X=E_7$ and assume $g^\sigma g^{-1}\in \emph N_X(T)$ maps to $w\in W$. Assume $\Psi$ and $\Psi'$ are two $p$-closed subsets of $\Phi$ of type $E_6$. If $w\in \emph N_W(W(\Psi))\cap \emph N_W(W(\Psi'))$, then either $w\in W(\Psi)\cap W(\Psi')$ or $w\notin W(\Psi)\cup W(\Psi')$.
\end{claim}

We first observe that Claim \ref{claim} implies $\mathcal T(x_1)\cap \mathcal T(x_2)=\emptyset$. Consider a maximal torus $S$ of $X_\sigma$; assume $S=(T_w)_\sigma$, where $T_w=T^g$ and $g^\sigma g^{-1}$ maps to $w\in W$. By Lemma \ref{reductive_overgroups_of_fixed_torus}, the closed connected reductive subgroups of $E_7$ containing $T_w$ are precisely the subgroups $K(\Psi)^{g}$, where $K(\Psi)=\gen{T, U_\alpha, \alpha \in \Psi}$ and $\Psi$ is $p$-closed and $w$-stable (i.e., $w\in \N_W(W(\Psi))$). Then Claim \ref{claim}, together with Lemma \ref{bijection_conjugacy_classes_W_Psi}, implies that $S=(T_w)_\sigma$ cannot be contained in both a maximal subgroup of type $^2\!E_6(q)_{\text{sc}}D_{q+1}$ and a Levi complement of type $E_6$, so that $S \notin \mathcal T(x_1)\cap \mathcal T(x_2)$ and (since $S$ was arbitrary) $\mathcal T(x_1)\cap \mathcal T(x_2)=\emptyset$.

In order to prove Claim \ref{claim}, we recall that $W=\gen{x}\times W^+$, where $|x|=2$ and $W^+\cong \Sp_6(2)$ is the ``rotation subgroup'', consisting of the elements of $W$ with determinant $1$ in the action on $\mathbf R^7$. 
We will view the elements of $W$ as pairs, according to this decomposition. If $\Psi$ is a subset of type $E_6$, then $W(\Psi)\cong \SO^-_6(2)$. Clearly we cannot have $W(\Psi)\leq W^+$, since $W(\Psi)$ contains reflections. Let $K$ be the unique subgroup of $W(\Psi)$ of index $2$, isomorphic to $\Omega^-_6(2)$. Then $K\leq W^+$. Let $H\cong \SO^-_6(2)$ be the normalizer of $K$ in $W^+$; we have $H=K\rtimes \gen{r}$, where $r$ is a reflection in a nonsingular vector (for the orthogonal geometry on $\F_2^6$). We have $W(\Psi)=\gen{(x,r), K}$ and $\N_W(W(\Psi))=\gen x \times H$. Now if $\Psi'$ is another subset of type $E_6$, we have $\Psi'=\Psi^g$ with $g \in W^+$, and consequently $W(\Psi')=\gen{(x,r^g), K^g}$ and $\N_W(W(\Psi'))=\gen x \times H^g$. 
We see that Claim \ref{claim} is equivalent to the following condition:
\begin{itemize}
    \item[($\star$)] Fix $\Psi$ as above. Then, for every $g\in W^+$, $W(\Psi)\cap \N_W(W(\Psi^g)) \leq W(\Psi^g)$.
\end{itemize}
It is easy to see that ($\star$) is equivalent to
\begin{itemize}
    \item[($\star\star$)] Fix $K\leq W^+$ and $H\leq W^+$ as above. Then, for every $g\in W^+$, $K\cap H^g\leq K^g$.
\end{itemize}
Here the condition in $(\star\star)$ holds in general, in the following sense. Assume $q$ is even, and recall that $\Sp_{2m}(q)\cong \SO_{2m+1}(q)$ (see Remark \ref{rem: convention} for further details concerning this isomorphism). Denote by $V$ the $(2m+1)$-dimensional orthogonal module. Then $(\star\star)$ is a particular case of the following lemma.

\begin{lemma}
Assume $W$ and $W'$ are nondegenerate hyperplanes of $V$ (not necessarily of the same sign). Then $\Omega(W)\cap \SOT(W')\leq \Omega(W')$.
\end{lemma}

\begin{proof}
Recall that $\Omega(W)$ can be characterized as the subset of $\SO(W)$ consisting of the elements $g$ such that $\text{dim}\,\C_W(g)$ is even (cf. \cite[p. 77]{Wil}). We have $V=W\perp V^\perp$, and $g$ acts trivially on $V^\perp$, therefore $\text{dim}\,\C_W(g)=\text{dim}\,\C_V(g)-1$, which is independent of $W$. This proves the lemma.
\end{proof}

Claim \ref{claim} is proved and Theorem \ref{theorem_1.5_restatement} follows in this case.


\subsection{$F_4(q)$}
\label{subsection_F4}
For exceptional groups, we proved Theorem \ref{theorem_1.5_restatement} in all cases except for $F_4(q)$, which is handled in this subsection. 

We first fix some notation taken from \cite{Law}. Let $\mathbf R^4$ be equipped with the usual orthonormal basis $e_1, \ldots, e_4$. We may take $\Phi \subseteq \mathbf R^4$ with set of positive roots
\[
\Phi^+=\{e_i\pm e_j, 1\leq i<j\leq 4\}\cup \{e_i, 1\leq i \leq 4\}\cup \{(e_1\pm e_2\pm e_3\pm e_4)/2\}
\]
and base
\[
\Sigma=\{e_2-e_3, e_3-e_4, e_4, (e_1-e_2-e_3-e_4)/2\}.
\]
As in \cite{Law}, we will write $1$ in place of $e_1$, $1-2$ in place of $e_1-e_2$, $+---$ in place of $(e_1-e_2-e_3-e_4)/2$, etc. The corresponding reflections in the Weyl group will be denoted by $w_1, w_{1-2}, w_{+---}$, etc. In \cite{Law} the complete list of maximal tori of $F_4(q)$ is given. In particular, for each $(\delta, \delta')\in \{+,-\}^2$, there are two conjugacy classes of maximal tori of order $(q^3 + \delta 1)(q + \delta' 1)$; we let $T_{\delta, \delta'}^i$, $i=1,2$, be representatives for the two classes (so for instance $T_{+,-}^1$ is a representative of a class of tori of order $(q^3+1)(q-1)$). Assume $T_{\delta, \delta'}^i=(T_{w})_\sigma$ with $w=w_{\delta, \delta'}^i$. 
With notation as in \cite[pp. 93--96]{Law}, we may choose
\begin{align*}
&w_{+,+}^1=w^{(13)}=w_3w_{2-3}w_{1-2}w_4 && w_{+,+}^2=w_{(20)}=w_1w_2w_4w_{+-+-}\\
&w_{+,-}^1=w^{(15)}=w_4w_{3-4}w_{2-3} && w_{+,-}^2=w_{(13)}=w_4w_{3-4}w_{+--+}\\
&w_{-,+}^1=w^{(14)}=w_1w_{3-4}w_{2-3} && w_{-,+}^2=w_{(15)}=w_{1-2}w_4w_{++--}\\
&w_{-,-}^1=w^{(12)}=w_{3-4}w_{2-3} && w_{-,-}^2=w_{(7)}=w_4w_{+---}
\end{align*}
Here composition is right-to-left; this however makes no difference, because in a Weyl group every element is conjugate to its inverse, cf. \cite[Corollary p. 45]{Car}. We are now ready to begin the proof of Theorem \ref{theorem_1.5_restatement} in this case. We divide the cases $q$ even and $q$ odd.

\textbf{(a)} Assume $q$ is odd. By \cite[Table 6]{GM2} we have that $x_1$ is contained only in a subgroup $^3\!D_4(q).3$; and by \cite[Table 1]{GM} $x_2$ is contained only in a subgroup $2.\Omega_9(q)$. We need to show that $\ca T(x_1)\cap \ca T(x_2)=\emptyset$. 
By order inspection, the only possibilities for $T(x_1)\cap \ca T(x_2)$ are the eight tori $T_{\delta, \delta'}^i$. By our choice (see \cite[pp. 94--95]{Law}) the maximal tori of type $1$ (i.e., the tori $T_{\delta, \delta'}^1$) are contained in $2.\Omega_9(q)$. This subgroup is obtained as the fixed points of a connected reductive subgroup of $F_4$ of type $B_4$. In order to conclude the proof, we need to show that none of the tori $T_{\delta, \delta'}^1$ belongs to a conjugate of $^3\!D_4(q).3$. 
Fix $(\delta,\delta')$, and fix $g\in X$ such that $T_{w_{\delta,\delta'}^1}=T^g$. There is a unique $p$-closed subset $\Psi$ of $\Phi$ of type $D_4$, namely the set of all long roots (the set of all short roots is only $2$-closed). Of course $\Psi$ is fixed by every element of $W$. Correspondingly, by Lemma \ref{reductive_overgroups_of_fixed_torus}, $T_{w_{\delta,\delta'}^1}$ has a unique connected reductive overgroup of type $D_4$, namely $\gen{T, U_\alpha, \alpha \in \Psi}^g$. The fixed points of such a subgroup is of type $D_4(q)$ or $^2\!D_4(q)$. Indeed, this is true for every maximal torus of $B_4$. It follows that $T_{\delta,\delta'}^1$ is contained in $D_4(q)$ or $^2\!D_4(q)$, but not in $^3\!D_4(q)$. This concludes the proof in case $q$ is odd.

\textbf{(b)} Assume $q$ is even. There is an automorphism $\gamma$ of $F_4(q)$ which induces a graph automorphism of order two on the Dynkin diagram, sending $2-3$ to $+---$ and $3-4$ to $4$. In this case, there are two conjugacy classes of maximal subgroups isomorphic to $\Om_9(q)$: we pick representatives $B^1_4$ and $B^2_4$ for them. Similarly, there are two conjugacy classes of maximal subgroups isomorphic to $^3\!D_4(q).3$, $\POm_8^+(q).S_3$, $P$, $Q$: we pick representatives $^3\!D^i_4$, $D^i_4$, $P^i$, $Q^i$, $i =1,2$, in the respective cases. Here $P^1$ (resp. $P^2$) denotes a parabolic subgroup of type $B_3$ (resp. $C_3$); $Q^1$ (resp. $Q^2$) denotes a parabolic subgroup of type $A_2\times \tilde{A_1}$ (resp. $\tilde{A_2}\times A_1$), where $\tilde{A_i}$ denotes a subset consisting of short roots. (The reason why there are two classes of reductive subgroups as above is that there are subsets of $\Phi$ of type $C_4$ and $D_4$ which are $2$-closed; see \cite[Proposition 3.15]{MT}.) We see that all the pairs of classes above are fused by $\gamma$, i.e., $(B^1_4)^\gamma = B^2_4$, and similarly for the others. There are also maximal subgroups $e.(\PSL_3^{\varepsilon}(q)\times \PSL_3^{\varepsilon}(q)).e.2$, with $\varepsilon \in \{+,-\}$ (one class for each sign; here $\PSL_3^+(q)=\PSL_3(q)$, $\PSL_3^-(q)=\PSU_3(q)$ and $e=(3,q-\varepsilon 1)$). We let $R_\varepsilon$ be representatives of these classes.

We observe that for every $(\delta, \delta')\in \{+,-\}^2$, $\gamma$ exchanges the classes of $T_{\delta, \delta'}^1$ and $T_{\delta, \delta'}^2$. This can be seen as follows. In \cite[Table 1]{Gut}, the action of $\gamma$ on $\Phi$ is computed (see also (2.15) of the same paper). In particular we have
\begin{align*}
&(w^1_{+,+})^\gamma = w_{3+4}w_{+---}w_2w_{3-4} \\
&(w^1_{+,-})^\gamma = w_{3-4}w_4w_{+---} \\
&(w^1_{-,+})^\gamma = w_{1+2}w_4w_{+---} \\
&(w^1_{-,-})^\gamma = w_4w_{+---}
\end{align*}
At this point one computes that in each case $(w^1_{\delta,\delta'})^\gamma$ is $W$-conjugate to $w^2_{\delta,\delta'}$. This is immediate if $(\delta, \delta')=(-,-)$. In general, it is sufficient to prove that $(w^1_{\delta,\delta'})^\gamma$ is not conjugate to $w^1_{\delta,\delta'}$. This can be done for instance by showing that $(w^1_{\delta,\delta'})^\gamma$ and $w^1_{\delta,\delta'}$ have different root lengths inside the $(-\delta' 1)$-eigenspace relative to the action on $\mathbf R^4$. It follows that $(T_{\delta, \delta'}^1)^\gamma$ is $F_4(q)$-conjugate to $T_{\delta, \delta'}^2$.


Now note that for every $(\delta, \delta')$, each subgroup $B^i_4$, $^3\!D^i_4$, $D^i_4$, $P^i$ and $Q^i$ contains members of at most one class of tori of type $T_{\delta,\delta'}$. 
By the same argument as in item (a), we may choose notation such that for every $(\delta, \delta')$, $T_{\delta, \delta'}^1$ belongs to a conjugate of $B^1_4$, $^3\!D^1_4$ and, possibly, $D^1_4$ (but not $D^2_4$). 
We now want to show that the maximal tori $T_{\delta, \delta'}^1$ can possibly belong only to conjugates of $P^1$ and $Q^1$, but not to conjugates of $P^2$ and $Q^2$. If we prove this, it will automatically follow that the tori $T^2_{\delta, \delta'}$ can belong only to type $2$ subgroups. Note that by the previous considerations, if $T_{\delta, \delta'}^1$ belongs to a conjugate of $P^1$ (resp. $Q^1$), then it does not belong to a conjugate of $P^2$ (resp. $Q^2$). 

By order considerations, $P^1$ and $P^2$ can contain the tori $T_{-,-}^i$ and $T_{+,-}^i$ for $i=1,2$. Moreover $Q^1$ and $Q^2$ can contain the tori $T_{-,+}^i$ and $T_{-,-}^i$ for $i=1,2$. No other embedding of the tori $T_{\delta, \delta'}^i$ in parabolic subgroups occurs. We see that $w_{-,-}^1=w^{(12)}$ and $w_{+,-}^1=w^{(15)}$ belong to the Weyl subgroup of type $B_3$ corresponding to removing $+---$ from the base $\Sigma$. Therefore 
we deduce that $T_{-,-}^1$ and $T_{+,-}^1$ belong to Levi complements of $F_4(q)$-conjugates of $P^1$.
Therefore the case $P$ is done. We move to case $Q$, which is similar. We have $w_{-,+}^1=w^{(14)}$. We see that both $w^{(12)}$ and $w^{(14)}$ belong to Weyl subgroups corresponding to a subset of type $A_2 \times \tilde{A_1}$. Indeed $w^{(12)}$ lies in the natural one corresponding to removing $4$ from the base $\Sigma$; and $w^{(14)}$ lies in the subset in which a base of $A_2$ is $\{2-3,3-4\}$ and a base of $\tilde{A_1}$ is $\{1\}$.
Therefore, we obtain that $T_{-,-}^1$ and $T_{-,+}^1$ belong to Levi complements of $F_4(q)$-conjugates of $Q^1$.

With all the information we have gathered, it is not difficult to complete the proof of Theorem \ref{theorem_1.5_restatement} in this case. With this aim, we choose a generator $x_{+,-}^i$ of the cyclic torus $T_{+,-}^i$, $i=1,2$, and a generator $x_{-,+}^i$ of the cyclic torus $T_{-,+}^i$, $i=1,2$. Moreover we choose elements $x_1$ and $x_2$ as in item (a); in particular $x_1$ belongs only to $^3\!D^1_4$ and $^3\!D^2_4$, and $x_2$ belongs only to $B^1_4$ and $B^2_4$. Our set of elements is therefore
\[
A_b=\{x_1, x_2, x_{+,-}^1, x_{+,-}^2, x_{-,+}^1, x_{-,+}^2\}.
\]
We now want to show that there is not a maximal torus of $F_4(q)$ belonging to an overgroup of all these elements. By order considerations, if a torus belongs to $\mathcal T(x_1)$ and $\mathcal T(x_2)$, then it must be one of the eight tori $T_{\delta,\delta'}^i$. Assume $i=1$: the argument is entirely symmetric and the case $i=2$ is handled in the same way. By our choice of notation, $x_{+,-}^2$ and $x_{-,+}^2$ belong to type $2$ subgroups. However 
$T_{\delta,\delta'}^1$ does not belong to any of these. The only other overgroups of maximal rank of $x_{+,-}^2$ and $x_{-,+}^2$ are, respectively, $R_-$ and $R_+$. However, our torus $T_{\delta,\delta'}^1$ belongs to exactly one of these (depending on the value of $\delta$). Therefore we have shown that the overgroups of our six elements cannot contain a common maximal torus, and the proof of Theorem \ref{theorem_1.5_restatement} (and therefore Theorem \ref{main_theorem_absolute_constant_2}(1)) for exceptional groups is concluded.

\section{Classical groups of bounded rank}

As promised at the beginning of Section \ref{section_bounded rank}, we immediately deal with $G\cong \PSL_2(q)$.

\begin{proof}[Proof of Theorem \ref{main_psl_2}]
Let $G=\PSL_2(q)$. The subgroup structure of $G$ is well known, cf. \cite[Chapter 3.6]{Suz}. Let $d=(2,q-1)$. Let $S_{\pm}$ be the set of elements of $G$ with order strictly larger than $5$ and dividing $(q\pm 1)/d$, and let $S=S_+\cup S_-$. We have $|S_\pm|/|G|=1/2+O(1/q)$. Let $F$ be the set of elements of $G$ lying in subfield subgroups; we have $|F|/|G|=O(1/q^{1/2})$ (we observed this in Remark \ref{remark_value_a}).

Let $A_b=\{x_1, x_2\}$, where $x_1$ has order $(q-1)/d$ and $x_2$ has order $(q+1)/d$. Every element of $S_+$ (resp. $S_-$) invariably generates with $x_1$ (resp. $x_2$). This proves Theorem \ref{main_theorem_absolute_constant_2}. Every element of $S_+ \setminus F$ (resp. $S_-\setminus F$) invariably generates with every element of $S_-$ (resp. $S_+$). This proves Theorem \ref{main_theorem_strengthening_bounded_rank}. Now consider $B=(S_+ \times S_-)\cup (S_- \times S_+)\subseteq G^2$; we have $|B|/|G|^2=1/2+O(1/q)$. Moreover $B\setminus F^2$ consists of invariable generating pairs, which proves Theorem \ref{main_theorem_two_random_elements_bounded_rank}.
\end{proof}
\begin{remark}
If $q$ is a square, every semisimple element of $\text{PGL}_2(q^{1/2})$ of order strictly larger than $5$ belongs to $S_-$. Hence, by Theorem~\ref{almost_all_connected_components_maximal_rank}, almost all elements of $\text{PGL}_2(q^{1/2})$ are contained in $S_-$. Therefore, in the above proof we only need to consider elements belonging to $\PSL_2(q^{1/r})$ with $r$ odd, and in Theorem \ref{main_theorem_strengthening_bounded_rank} we could get an error term of type $O(1/q^{2/3})$; but we do not insist on this.
\end{remark}

In the remainder of this section, we will prove Theorem \ref{theorem_1.5_restatement} for classical groups. Together with Section \ref{subsection_exceptional_groups}, this will conclude the proof of Theorem \ref{theorem_1.5_restatement}, which in turn implies Theorem \ref{main_theorem_absolute_constant_2}(1).

We first make our choice for the type of $X$. Let $k$ be an algebraic closure of a finite field of characteristic $p$ and let $X$ be one of the simple algebraic groups $\SL_n(k)$ where $n\ge 2$, $\Sp_n(k)$ where $n\ge 4$, or $\SO_n(k)$ where $n\ge 7$. We require that if $X=\SO_n(k)$ and $p=2$ then $n$ is even.

Denote by $V= k^n$ the natural module of $X$. Here $\Sp_n(k)$ is the group of isometries of a nondegenerate bilinear alternating form on $V$ ($n$ is even), while $\SO_n(k)$ is the connected component of the isometry group $\GO_n(k)$ of a quadratic form on $V$, with associated nondegenerate bilinear form. 
We have $|\GO_n(k):\SO_n(k)|=2$. 
These groups are well defined up to conjugation in $\GL_n(k)$ since all such forms are equivalent (for all this, see for instance \cite[Section 1.2 and Definition 1.15]{MT}).

Let $\sigma : X\rightarrow X$ be a Steinberg morphism, as in \cite[p. 434]{LSe}, such that $X_{\sigma}$ is one of the following finite groups:
\[
X_{\sigma}=\SL_n(q), \SU_n(q), \Sp_n(q) \, (n \, \text{even}), \SO^{\varepsilon}_{n}(q)\, (q \, \text{odd}), \Omega^{\pm}_n(q) \, (q, n \, \text{even})
\]
where $\varepsilon=\pm$ for $n$ even and $\varepsilon=\circ$ for $n$ odd, where by convention $\SO^\circ_n(q) = \SO_n(q)$. Specifically, $\sigma=\thickbar{\sigma}\tau$, where $\thickbar \sigma$ is a Frobenius morphism corresponding to the field automorphism $\alpha \mapsto \alpha^q$ of $k$ (for $q$ a power of $p$), and $\tau=1$, or $X=\SO_n(k)$ and $\tau$ is conjugation by a reflection in a nonsingular vector, or $X=\SL_n(k)$ and $\tau$ is the inverse-transpose map (all this with respect to certain fixed bases).

Except for $\SO^{\varepsilon}_n(q)$, which has a derived subgroup $\Omega^{\varepsilon}_n(q)$ of index $2$, and except for finitely many other cases, the group $X_\sigma$ is perfect. 
See \cite[Chapter 2]{KL} for the definition of $\Omega^{\varepsilon}_n(q)$.

\label{subsection_classical_groups}
\subsection{Subgroups of maximal rank in classical groups}
We need to understand the subgroups of maximal rank in $X_\sigma$. 
The proof of the following theorem is largely taken from \cite{LSe}, with additional claims from \cite[Section 13]{MT}. We prefer to give a proof, since we have not found it in the literature. 

\begin{theorem}
\label{subgroups_maximal_rank_algebraic}
Let $X$ and $\sigma$ be as above. Let $M$ be a $\sigma$-stable closed subgroup of $X$ of maximal rank. Then, $M$ is contained in a $\sigma$-stable subgroup of the following types.
\begin{enumerate}
    \item[(i)] Stabilizer in $X$ of a nonzero proper subspace of $V$. If $X=\SpT_n(k)$ or $\SOT_n(k)$, the space is totally singular or nondegenerate. If it is nondegenerate, it can be chosen of even dimension.
    \item[(ii)] In case $X=\SLT_n(k)$, stabilizer in $X$ of a pair of proper subspaces $U$ and $W$ such that $\dimT U + \dimT W = \dimT V$, $\dimT U \neq \dimT W$ and either $U\leq W$ or $U\cap W =0$.
    \item[(iii)] Stabilizer in $X$ of a decomposition $V=V_1\oplus \cdots \oplus V_t$ with $t\geq 2$. If $X=\SpT_n(k)$ or $\SOT_n(k)$, the spaces $V_i$ are isometric, and either pairwise orthogonal and nondegenerate of even dimension, or $t=2$ and the spaces are totally singular.
    \item[(iv)] $p=2$, $X=\SpT_n(k)$ and $M\leq \emph N_X(\SOT_n(k))=\GO_n(k)$. 
\end{enumerate}
\end{theorem}

\begin{proof}
We assume for the first part of the proof that we are not in case $X=\SL_n(k)$ with $\sigma$ involving the inverse-transpose map. In particular $\sigma$ can be regarded as a semilinear map of $V$. 

Assume first $M$ fixes a proper nonzero $\sigma$-stable subspace $W$, and choose it to be of minimal dimension. If there is a form, then $M$ fixes also $W^{\perp}$, which is $\sigma$-stable (cf. \cite[Proposition 2.5]{LSe}). Then $M$ fixes $W\cap W^{\perp}$, which is $\sigma$-stable, hence by minimality $W$ is either nondegenerate or totally isotropic. If the space is nondegenerate of odd dimension $\ell$, the only possibility is $X=\SO_n(k)$ and $M^{\circ}\leq \SO_\ell(k)\times \SO_{n-\ell}(k)$. If $n$ is odd, then $M$ stabilizes also a nondegenerate space of even dimension $n-\ell$. If $n$ is even, instead, the stabilizer of $W$ has rank $n/2-1$, hence it is not of maximal rank. Assume now $W$ is totally isotropic. Then $M$ fixes also the set of singular vectors of $W$, which is $\sigma$-stable, hence by minimality either $W$ is totally singular, or $X=\SO_n(k)$, $p=2$, $n$ is even and $W$ is a nonsingular $1$-space. In the latter case, however, the stabilizer is isomorphic to $\Sp_{n-2}(k)$, 
which is not of maximal rank. In particular, if $M$ fixes a proper nonzero $\sigma$-stable subspace, we are in case (i) of the statement.

Then we assume that $M$ does not fix any proper nonzero $\sigma$-stable subspace of $V$. Let $H:=M^{\circ}$ be the connected component of $M$. 
The proof of \cite[Lemma 3.2]{LSe} shows that either both $M$ and $H$ act homogeneously, or we are in case (iii) of the statement. The parity requirement in (iii) comes from the following reason: if $X=\SO_n(k)$ and the decomposition $V=V_1\perp \cdots \perp V_t$ is isometric, then the connected component of the stabilizer of the decomposition is $\SO(V_1)\times \cdots \times \SO(V_t)$, which has maximal rank only if $V_i$ has even dimension.

Therefore assume $M$ and $H$ act homogeneously. 
By assumption $M$, hence $H$, contains a maximal torus $S$. Using the explicit description of maximal tori in classical groups, we note that as a $kS$-module $V$ is the sum of $1$-dimensional pairwise nonisomorphic modules. In particular, it follows that both $M$ and $H$ act irreducibly.

Since $H$ is connected and it acts faithfully and irreducibly on $V$, it follows that $H$ is reductive (cf. \cite[Proposition 15.1]{MT}). Then $H=[H,H]\Z(H)^{\circ}$. 
It follows by Schur's lemma that $\Z(H) \leq \Z(X)$, which is a finite group, hence $\Z(H)^{\circ} \leq \Z(X)^{\circ}=1$. In particular, $H=[H,H]$ is semisimple. 
In \cite[Chapter 13]{MT}, $H$ is called a subsystem subgroup of $X$. One checks easily that the examples in \cite[Theorem 13.12]{MT} give (well recognizable) reducible subgroups. 
In \cite[Proposition 13.15]{MT}, item (1) corresponds to item (iv) in this theorem, and 
item (2) does not arise by assumption (if $X=\SO_n(k)$ and $p=2$ then $n$ is even).
By \cite[Theorem 13.14]{MT}, there are no other possibilities for $H$. This concludes the proof, except when $X=\SL_n(k)$ and $\sigma$ involves the inverse-transpose map (the previous proof works also if $\sigma$ is the identity).

Let us consider the remaining case. As in the proof of \cite[Lemma 3.7]{LSe}, we view $\SL_n(k)$ as a subgroup of $\SO_{2n}(k)$: we may decompose the orthogonal module as $E\oplus F$, in such a way that the embedding of $\SL_n(k)$ is given by $g \mapsto \text{diag}(g,g^{-T})$. 
It follows from the proof of \cite[Lemma 3.7]{LSe} that either we are in cases (i), (ii) or (iii) of the statement, or $M$ and $M^\circ$ act homogeneously on $E$. Then we may proceed exactly as in the first part of the proof (once $M^\circ$ was shown to be homogeneous, the morphism $\sigma$ was not used anymore). 
\end{proof}

\subsection{Descending to finite groups}
Aschbacher \cite{Asc} described the core-free maximal subgroups of the finite almost simple classical groups, dividing them into nine classes, denoted by $\ca C_1, \ldots, \ca C_8, \ca S$. We refer the reader to \cite{KL} for a detailed description of the first eight classes. (Although in the present paper this does not make any difference, we remark that classes in \cite{Asc} and classes in \cite{KL} differ slightly; we take \cite{KL} as a reference, modulo the change explained in Convention \ref{convention_aschbacher}, below.)

Consider the following groups:
\begin{equation}
\label{eq:groups}
\GL_n(q), \GU_n(q), \Sp_n(q) \, (n \, \text{even}), \SO^{\varepsilon}_{n}(q)\, (q \, \text{odd}), \Omega^{\pm}_n(q) \, (q, n \, \text{even})
\end{equation}
with $\varepsilon \in \{+,-,\circ\}$. Let $H$ be a group as in \eqref{eq:groups}. 
We will refer to a group $G$ with $H'\leq G \leq H$ as a \textit{finite classical group}.

Let $G$ be a finite classical group with $G'$ quasisimple. Note that if $M$ is a maximal subgroup of $G$ not containing $G'$, then $\Z(G) \leq M$, and in particular $M$ corresponds to a core-free maximal subgroup of the almost simple group $G/\Z(G)$, so Aschbacher's description applies to $M$.

For a finite classical group $G$ with $G'$ quasisimple, set $u=2$ if $G'=\SU_n(q)$ and $u=1$ otherwise. We give here a rough description of the maximal subgroups of $G$ of classes $\ca C_1, \ca C_2, \ca C_3$. Maximal subgroups from class $\ca C_1$ are subspace stabilizers; maximal subgroups from class $\ca C_2$ stabilize suitable direct sum decompositions of the natural module, so they are subgroups of $\GL_\ell(q^u)\wr S_{n/\ell} < \GL_n(q^u)$ for some $\ell< n$; and maximal subgroups from class $\ca C_3$ preserve an extension field structure on the natural module, so they are subgroups of $\GL_{n/b}(q^{ub})\rtimes \text{Gal}(\F_{q^{ub}}/\F_{q^{u}}) < \GL_n(q^u)$ for some prime $b$.

It is convenient for us to adopt the following convention throughout.

\begin{convention}
\label{convention_aschbacher}
When $q$ and $n$ are even, we consider the maximal subgroups $\SO_n^\pm(q) < \Sp_n(q)$ as members of Aschbacher's class $\ca C_1$.
\end{convention}

\begin{remark}\label{rem: convention}
Convention \ref{convention_aschbacher} is justified by the following discussion. When $q$ is even, $\Sp_{2m}(q) \cong \SO_{2m+1}(q)=\mathrm O_{2m+1}(q)$, the group of isometries of a nonsingular quadratic form $Q$ on a $(2m+1)$-dimensional space $V$. (Here, by nonsingular we mean that there are no nonzero vectors $v$ of $V^{\perp}$ such that $Q(v)=0$.) Under these assumptions, it turns out that $V^{\perp}$ is a $1$-dimensional subspace of $V$. Since $q$ is even and $Q$ does not vanish on $V^\perp$, it follows that $\SO_{2m+1}(q)$ centralizes $V^\perp$. Through this identification, the subgroups $\SO^{\pm}_{2m}(q)$ correspond to stabilizers of nondegenerate hyperplanes of $V$ (i.e., complements of $V^{\perp}$) of plus or minus type.
\end{remark}

In the following statement, $\text{Cl}_n(q)$ denotes a group as in \eqref{eq:groups} (this is unusual notation; however it will not be used elsewhere and it should not cause any confusion). When we write in brackets ``class $\ca C_i$'', we mean that the subgroup $M_\sigma$ under consideration is contained in a maximal subgroup of $X_\sigma$ of class $\ca C_i$.

\begin{theorem}
\label{classes_c1_c2_c3}
Let $X$, $\sigma$ and $M$ be as in Theorem \ref{subgroups_maximal_rank_algebraic}. Then, $M_{\sigma}$ is contained in a maximal subgroup of $X_{\sigma}$ from classes $\ca C_1$, $\ca C_2$ or $\ca C_3$. 

Now assume that $M$ is contained in the stabilizer of a decomposition $V=V_1\oplus \cdots \oplus V_t$ as in Theorem \ref{subgroups_maximal_rank_algebraic}(iii).
\begin{enumerate}
    \item If $X_{\sigma}=\SLT_n(q)$ or the decomposition is nondegenerate, then either $M^{\circ}_{\sigma}\leq \emph{Cl}_\ell(q)^t \cap X_{\sigma}$ (class $\ca C_2$), or $M^{\circ}_{\sigma}\leq \emph{Cl}_{n/b}(q^b)\cap X_{\sigma}$ (class $\ca C_3$), or $X_{\sigma}=\SUT_n(q)$ with $n$ even and $M^{\circ}_{\sigma}\leq \GLT_{n/2}(q^2)\cap X_\sigma$ (class $\ca C_2$). In case $X_{\sigma}$ is orthogonal, $\ell$ and $n/b$ must be even.
    \item If the decomposition is totally singular, then $X=\SpT_n(k)$ or $\SOT_n(k)$, $n$ is even, and either $M^{\circ}_{\sigma}\leq \GLT_{n/2}(q)$ (class $\ca C_2$), or $M^{\circ}_{\sigma}\leq \GUT_{n/2}(q)$ (class $\ca C_3$). 
\end{enumerate}

\end{theorem}

\begin{proof}

Note that if $M$ stabilizes $V=V_1\oplus \cdots \oplus V_t$ as in Theorem \ref{subgroups_maximal_rank_algebraic}(iii), then $M^\circ$ fixes each $V_i$ (indeed the subgroup stabilizing each $V_i$ is a closed subgroup of finite index of $M$, hence contains the connected component).
What is more, if $X=\SO_n(k)$ and the $V_i$ are nondegenerate then $M^\circ \leq \SO(V_1)\times \cdots \times \SO(V_t)$. Keeping in mind this observation, the proof of the theorem follows from the arguments in \cite[Section 4]{LSe}, together with Theorem \ref{subgroups_maximal_rank_algebraic}. 
\end{proof}

Before we define the set $A_b$ from Theorem \ref{theorem_1.5_restatement}, for the reader's convenience we recall some facts about maximal tori in classical groups. We begin with some observations regarding conjugacy classes in wreath products, which we will use also later in the paper (see Remark \ref{rem:weyl_bm}).

\begin{remark}
\label{rem:conjugacy_classes_wreath_products}
Let $H$ be a finite group and let $G=H\wr S_m$, where $(x_1, \ldots, x_m)^\tau = (x_{1\tau^{-1}}, \ldots, x_{m\tau^{-1}})$ for $(x_1, \ldots, x_m)\in H^m$ and $\tau \in S_m$. It is well known and easily seen that conjugacy classes in $G$ are in bijection with  partitions of $m$, where each part is labelled by a conjugacy class of $H$ (see, for example, \cite[Section 4.2]{james_kerber}).

Moreover, it follows from a straightforward calculation that, given an element $x=(x_1, \ldots, x_m)\tau \in H\wr S_m$, with $(x_1, \ldots, x_m)\in H$ and $\tau \in S_m$, one can identify its class as follows. Let $c_1, \ldots, c_t$ be the cycles of $\tau$; assume that $c_i=(i_1, \ldots, i_{\ell_i})$, so in particular $m=\sum_i^t \ell_i$. Then, the class of $x$ corresponds to the labelled partition of $m$ with parts $p_1, \ldots, p_t$, of length $\ell_1, \ldots, \ell_t$ respectively, where the part $p_i$ is labelled by the $H$-class containing $\prod_{j=1}^{\ell_{i}} x_j$. 

Consider now the special case $H=C_2$, so $G=C_2\wr S_m\cong W(B_m)$. In particular, conjugacy classes of $G$ are in bijection with partitions of $m$, where each part is equipped with a sign $\pm$ (we will refer to these as \textit{signed partitions}).

Let $K$ be the subgroup of $G$ consisting of the elements $(x_1, \ldots, x_m)\tau$ with $x_i=\pm$ and $\prod_{i=1}^m x_i = +$; in particular, $|G:K|=2$ and $K\cong W(D_m)$. The $W(B_m)$-classes of elements of $K$ are in bijection with signed partitions of $m$ where the product of the signs is $+$. Moreover, for $x\in K$, it follows from an easy computation that $x^K\neq x^G$ if and only if $x^G$ corresponds to a signed partition in which all parts have even length and have plus sign.
\end{remark}

\subsection{Maximal tori in classical groups}
\label{subsec:maximal_tori}
The conjugacy classes of maximal tori in finite classical groups have an interpretation in terms of (signed) partitions. We quickly recall some facts. For more details, see for instance \cite[Section 5]{FG2}. In each of the cases below, $\sigma$ denotes the defining Steinberg endomorphism.

\textbf{(i)} Let $G=\SL_n(q)$. We have $W(A_{n-1})\cong S_n$, so conjugacy classes of maximal tori in $G$ are in bijection with partitions of $n$. The torus $(T_w)_\sigma$ corresponding to a partition $w=(a_1, \ldots, a_t)$ of $n$ fixes a decomposition $V=V_1\oplus \cdots \oplus V_t$, acting irreducibly on the $a_i$-th dimensional space $V_i$ for every $i$. 

\textbf{(ii)} Let $G=\SU_n(q)$. As in (i), conjugacy classes of maximal tori in $G$ are in bijection with partitions of $n$.  The torus $(T_w)_\sigma$ corresponding to a partition $w=(a_1, \ldots, a_t)$ of $n$ fixes a decomposition $V=V_1\perp \cdots \perp V_t$, where $V_i$ is nondegenerate and of dimension $a_i$. If $a_i$ is odd then $(T_w)_\sigma$ acts irreducibly on $V_i$; if $a_i$ is even then $(T_w)_\sigma$ fixes $V_i=A_i\oplus B_i$, where $A_i$ and $B_i$ are totally singular (of dimension $a_i/2$), and $(T_w)_\sigma$ acts irreducibly on both. 

\textbf{(iii)} Let $G=\Sp_{2m}(q)$. Since $W(C_m) \cong C_2 \wr S_m$, by Remark \ref{rem:conjugacy_classes_wreath_products} conjugacy classes of maximal tori in $G$ are in bijection with signed partitions of $m$. The torus $(T_w)_\sigma$ corresponding to a signed partition $w=(a_1^{\varepsilon_1}, \ldots, a_t^{\varepsilon_t})$ of $m$, with $\varepsilon_i \in \{+,-\}$, fixes a decomposition $V=V_1\perp \cdots \perp V_t$, where $V_i$ is nondegenerate and of dimension $2a_i$. If $\varepsilon_i=-$ then $(T_w)_\sigma$ acts irreducibly on $V_i$, while if $\varepsilon_i=+$ it acts irreducibly on two complementary totally singular subspaces.

\textbf{(iv)} Let $G=\SO_{2m+1}(q)$. Since $W(B_m) \cong W(C_m)$, the conjugacy classes of maximal tori of $G$ are in bijection with signed partitions of $m$. The torus $(T_w)_\sigma$ corresponding to a signed partition $w=(a_1^{\varepsilon_1}, \ldots, a_t^{\varepsilon_t})$ of $m$, with $\varepsilon_i \in \{+,-\}$, fixes the decomposition $V=V_1\perp V_2$, where $\mathrm{dim}(V_1)=1$, $(T_w)_\sigma$ centralizes $V_1$, and acts on $V_2$ as explained for symplectic groups in (iii). (We note that this discussion holds also for $q$ even, recalling Remark \ref{rem: convention}.)

\textbf{(v)} Let $G=\SO^+_{2m}(q)$ ($q$ odd) or $G=\Omega^+_{2m}(q)$ ($q$ even). Then the conjugacy classes of maximal tori in $G$ are in bijection with the conjugacy classes of $W(D_m)$. Recall from Remark \ref{rem:conjugacy_classes_wreath_products} that a $W(B_m)$-class splits into two $W(D_m)$-classes if and only if it corresponds to a signed partition where all parts have even length and plus sign. Let $w$ be such a signed partition, and let $T_1$ and $T_2$ be corresponding non-conjugate maximal tori of $G$. Consider an embedding $G\leq \SO_{2m+1}(q)=S$, where $\F_q^{2m+1} = V= V_1\perp V_2$, $\mathrm{dim}(V_1)=1$, and $G\leq  \mathrm O(V_2)\cong \mathrm O^+_{2m}(q)$, centralizing $V_1$ (see Remark \ref{rem: convention} for the case $q$ even, where we have $V_1=V^\perp$). Then, $T_1$ and $T_2$ are maximal tori of $S$, corresponding to the same signed partition $w$, hence they are conjugate in $S$. More precisely, it is easy to see that they must be conjugate by an element $(-1,g)$, where $g\in \mathrm O(V_2)\cong \mathrm O^+_{2m}(q)$ (and $-1=1$ if $q$ is even). In particular, despite being non-conjugate in $G$, $T_1$ and $T_2$ are conjugate in $\mathrm O^+_{2m}(q)$, and therefore act on the orthogonal module in the same way. Therefore, for many purposes it is sufficient to look at the underlying signed partition of a maximal torus of $G$. Accordingly, we will write $(T_w)_\sigma$ to denote any representative of a conjugacy class of maximal tori in $G$ corresponding to a $W(D_m)$-class whose underlying signed partition is $w$. With this convention, the torus $(T_w)_\sigma$ acts as explained in (iii). (Recall by Remark \ref{rem:conjugacy_classes_wreath_products} that $w=(a_1^{\varepsilon_1}, \ldots, a_t^{\varepsilon_t})$ with $\prod_{i=1}^t \varepsilon_i =+$.)

\textbf{(vi)} Let $G=\SO^-_{2m}(q)$ ($q$ odd) or $G=\Omega^-_{2m}(q)$ ($q$ even). It turns out that the conjugacy classes of maximal tori of $G$ are in bijection with the $W(D_m)$-classes in the nontrivial coset of $W(D_m)$ in $W(B_m)$. These are clearly $W(B_m)$-classes; in particular, the conjugacy classes of maximal tori in $G$ are in bijection with signed partitions $w=(a_1^{\varepsilon_1}, \ldots, a_t^{\varepsilon_t})$ with $\prod_{i=1}^t \varepsilon_i =-$, and $(T_w)_\sigma$ acts as explained in (iii).


\subsection{Definition of the set $A_b$}
\label{subsection_definition_elements_classical}
Assume our finite classical group has natural module $V$ of dimension $n\geq 2$, defined over the field with $q^u$ elements, with $u=2$ in case of unitary groups, and $u=1$ otherwise. In view of Theorem \ref{main_psl_2}, we can also assume $n\geq 3$ (although this will not be relevant for the argument). In light of various isomorphisms between groups of small rank (see, for example,  \cite[Proposition 2.9.1]{KL}), we assume $n\geq 3$ for unitary groups, $n\geq 4$ for symplectic groups and $n\geq 7$ for orthogonal groups. 
Recall also that, for orthogonal groups, if $n$ is odd we have $q$ odd. 

We define $A_b$ as the set of elements appearing in Table \ref{table_bounded_rank}. 
More precisely, each element $x\in A_b$ is semisimple, and the corresponding entry in the table denotes the conjugacy class of a maximal torus containing $x$. (For the element $x_1$ in $\Omega^+_{2m}(q)$ with $m\geq 4$ even, see the convention in (v) of Subsection \ref{subsec:maximal_tori}.)

Of course, the torus alone does not determine uniquely the element (not even its order).

We require that the element $x$ has the following order on each irreducible fixed space. 
For convenience we will identify the spaces with the corresponding parts of the partition, as explained above.

We deal separately with linear and unitary groups, as we need to take care of the determinant. In $\SL_n(q)$, $x_1$ has order $(q^n-1)/(q-1)$, and $x_2$ has order $q^{n-1}-1$. In $\SU_n(q)$, $x_1$ has order $(q^n+1)/(q+1)$ for $n$ odd, and order $(q^n-1)/(q+1)$ for $n$ even; while $x_2$ has order $q^{n-1}-1$ for $n$ odd, and order $q^{n-1}+1$ for $n$ even. 

If $x\in \Sp_{2m}(q)$, then $x$ has order $q^a+1$ on each $(a^-)$, and order $q^a-1$ on each $(a^+)$. If $x\in \Omega^{\pm}_{2m}(q)$, then $x$ has order $(q^a+1)/(2,q-1)$ on each $(a^-)$, and order $(q^a-1)/(2,q-1)$ on each $(a^+)$. If $x\in \Omega_{2m+1}(q)$ the same holds; recall that $x$ also centralizes a $1$-space. 


We impose two further conditions. The element $x_3$ in $\Sp_4(q)$ with $q$ even acts as $(1^-,1^-)$. Accordingly, we require $x_3=(g,g^2)$ with $g$ of order $q+1$. Similarly, the element $x_3$ in $\Omega^+_8(q)$ acts as $(2^-, 2^-)$, and we have $x_3=(g,g^2)$ with $g$ of order $(q^2+1)/(2,q-1)$.  

With these choices, it is not difficult to check that, if $q$ is sufficiently large ($q\geq 11$), then every $x\in A_b$ is \textit{separable}, i.e., it has distinct eigenvalues on the natural module of the ambient algebraic group $X$. In Lemma \ref{order_x_to_the_f} below we will prove a little more.

\begin{table}
\small
\centering
\caption{$A_b=\{x_1,x_2,x_3,x_4\}$ in Theorem \ref{main_theorem_absolute_constant_2} for classical groups.}
\label{table_bounded_rank}       
\begin{tabular}{lllll}
\hline\noalign{\smallskip}
$G$ & $x_1$ & $x_2$ & $x_3$ & $x_4$ \\
\noalign{\smallskip}\hline\noalign{\smallskip}
$\SL_n(q)$, $n\geq 2$ & $n$ &  $n-1,1$ & & \\
$\SU_n(q)$, $n\geq 3$ &  $n$ & $n-1,1$ & &  \\
$\Om_{2m+1}(q)$, $m\geq 3$, $q$ odd & $m^-$  & $m^+$ & & \\
$\Om_{2m}^-(q)$, $m\geq 4$ & $m^-$  & $(m-1)^-,1^+$ & & \\
$\Sp_{2m}(q)$, $m\geq 2$ even, $q$ odd & $m^-$  & $(m-1)^-,1^+$ & & \\
$\Sp_{2m}(q)$, $m\geq 3$ odd, $q$ odd & $m^-$ & $(m-1)^-,1^-$ & $m^+$ & \\
$\Sp_{2m}(q)$, $m\geq 2$, $q$ even & $m^-$ & $(m-1)^-,1^+$ & $(m-1)^-,1^-$ & $m^+$ \\
$\Om_{2m}^+(q)$, $m\geq 5$ odd & $m^+$ & $(m-1)^-,1^-$ & & \\
$\Omega_{2m}^+(q)$, $m\geq 4$ even & $m^+$ & $(m-1)^-,1^-$  & $(m-2)^-,2^-$ & $(m-2)^-,1^-,1^+$ \\
\noalign{\smallskip}\hline
\end{tabular}
\end{table}

\subsection{Focusing on $\ca C_1$ and $\ca C_3$} By Theorem \ref{classes_c1_c2_c3}, $\ca M$ consists of members from Aschbacher's classes $\ca C_1, \ca C_2, \ca C_3$ (recall Convention \ref{convention_aschbacher}). We will show that, in fact, we can focus on overgroups of $x\in A_b$ from classes $\ca C_1$ and $\ca C_3$ (see Lemma \ref{classical_groups_we can_use_m_con} and the comments following it), which is a quite useful fact.

\begin{lemma}
\label{order_x_to_the_f}
Let $f$ be a positive integer. If $q\geq 11 f$ and $x \in A_b$, then $x^f$ is separable.
\end{lemma}

\begin{proof}
Let $x\in A_b$. As observed at the end of Subsection \ref{subsection_definition_elements_classical}, since $q\geq 11$ we have that $x$ is separable, hence contained in a unique maximal torus, which we denote by $T$. We need to prove two things:

\begin{itemize}
    \item[(i)] if $W\leq V$ is irreducible for $T$, then $W$ is irreducible for $x^f$, and
    \item[(ii)] if $W_1$ and $W_2$ are distinct irreducible modules for $T$, then $x^f$ has distinct eigenvalues on $W_1$ and $W_2$. 
\end{itemize}

(Note that, since $x\in T$ is separable, any two distinct irreducible $T$-modules $W_1$ and $W_2$ are not isomorphic.) Items (i) and (ii) imply that $x^f$ has the same fixed spaces as the torus $T$, and therefore $x^f$ is separable. The argument is essentially the same in all cases.

For item (i), we show that if $q\geq 11 f$ then the order of $x^f$ is large enough to imply that $x^f$ acts irreducibly on $W$. For instance, assume the torus $T$ acts irreducibly on a nondegenerate module $W$ of dimension $2a$ (and assume we are not in the unitary case). Then by our choices the order of $x$ on $W$ is larger than $(q^a+1)/2$. Then the order of $x^f$ is larger than $(q^a+1)/2f$, which is strictly larger than $q^{a-1}+1$; in particular $x$ does not fix any proper nondegenerate submodule of $W$. What is more, if $x^f$ fixes a proper totally singular submodule $U$, then this has dimension $\ell \leq a$. If $\ell < a$, the same argument as above gives a contradiction, and if $\ell=a$, then the order of $x^f$ would divide both $q^a-1$ and $q^a+1$, hence it would divide $2$, which is false. 
In unitary groups, the argument is the same; and in case $W$ is totally singular, the argument is similar (we note also that the dimension of any $\gen{x^f}$-submodule of $W$ divides the dimension of $W$).

For item (ii), we first observe that if $W_1$ and $W_2$ have different dimensions, the claim is obvious. There are cases that can be checked separately, namely $x_2$ in $\SL_2(q)$, $x_2$ in $\SU_3(q)$, $x_3$ in $\Sp_4(q)$ with $q$ even, and $x_3$ in $\Omega_8^+(q)$. In all other cases, if $W_1$ and $W_2$ have equal dimension, say $a$, then they are totally singular and $T$ acts irreducibly on both, with $W_1 \oplus W_2$ nondegenerate. Let us assume we are not in the unitary case. If $x$ has eigenvalues $\{\lambda, \lambda^q, \ldots, \lambda^{q^{a-1}}\}$ on $W_1$, then it has eigenvalues $\{\lambda^{-1}, \lambda^{-q}, \ldots, \lambda^{-q^{a-1}}\}$ on $W_2$. By our choices, $|\lambda| \geq (q^a-1)/2$. Since, by (i), $x^f$ acts irreducibly on both $W_1$ and $W_2$, the eigenvalues of $x^f$ on $W_1$ and $W_2$ either coincide (as sets) or are disjoint. Assume for a contradiction that they coincide. In particular, $\lambda^f= \lambda^{-fq^{i}}$ for some $0\leq i\leq a-1$. Then $|\lambda|$ divides $f(q^i+1) \leq f(q^{a-1}+1)$, which contradicts $q\geq 11f$. The unitary case is similar (in this case the eigenvalues of $x$ on $W_2$ are the $q$-th powers of the inverses of the eigenvalues on $W_1$).
\end{proof}

\begin{lemma}
\label{classical_groups_we can_use_m_con}
Assume $q\geq 11n^4$, take $x\in A_b$ and let $T$ be the corresponding maximal torus in $X_\sigma$. Then for $M_\sigma\in \ca M(x)$ we have $x\in T\leq M_\sigma ^\circ$. In particular, $\ca T(x)$ consists of the maximal tori contained in some overgroup of $x$ belonging to $\ca M_{\text{con}}$. 
\end{lemma}

\begin{proof}
By Theorems \ref{subgroups_maximal_rank_algebraic} and \ref{classes_c1_c2_c3}, and by the fact that $x$ is separable, 
we see that the result is easily established unless $M_\sigma$ preserves a direct sum decomposition or is an extension field subgroup.

Note that $x$ fixes at most four irreducible spaces on the natural module. In particular, if $x\in M_\sigma\leq \GL_{n/t}(q^u)\wr S_t$, then $x$ induces a permutation in $S_t$ having at most four cycles, and therefore having order at most $n^4$. It follows that $x^f\in M^\circ_\sigma$ for some $f\leq n^4$. If $x\in M_\sigma$ and $M_\sigma$ is an extension field subgroup, then the index of $M^\circ_\sigma$ in $M_\sigma$ is at most $n$. 
In particular, $x^f\in  M^\circ_\sigma$ for some $f\leq n$.

Since $q\geq 11n^4 \geq 11f$, by Lemma \ref{order_x_to_the_f} $x^f$ is separable, and in particular regular semisimple. Therefore, 
the maximal torus corresponding to $x^f$, which is $T$, is contained in $M^\circ_\sigma$.
\end{proof}

The conclusion to Lemmas~\ref{order_x_to_the_f} and \ref{classical_groups_we can_use_m_con} hold under slightly weaker conditions on $q$, but this is not required for the proof of Theorem \ref{theorem_1.5_restatement}, where we assume $q\geq Cr^r$. 

A very useful consequence of Lemma \ref{classical_groups_we can_use_m_con} is that, for each element of $A_b$, we just need to determine the overgroups of its maximal torus contained in the Aschbacher's classes $\ca C_1$ and $\ca C_3$; in class $\ca C_3$ we only have to consider the linear subgroup $\GL_{n/b}(q^{ub}) < \GL_n(q^u)$. 
We will now determine such overgroups.

\subsection{Overgroups of the elements of $A_b$}
\label{subsection_overgroups_of_elements_classical_groups}

The overgroups from class $\ca C_1$ (except for $\SO_n^\pm (q) < \Sp_n(q)$) are easily determined by looking at Table \ref{table_bounded_rank}, since the elements are separable. We now deal with the remaining case. We have $M_\sigma=\SO^{\pm}_{2m}(q) < \Sp_{2m}(q)\cong \SO_{2m+1}(q)$ with $q$ even,
and we have  $M^\circ_\sigma=\Omega^\pm_{2m}(q)$. 

\begin{lemma}
\label{stabilizers_hyperplanes_maximal torus}
Let $q$ be even, let $\sigma$ be the defining Steinberg endomorphism of $G=\mathrm{SO}_{2m+1}(q)$, and let $w$ be a signed partition of $m$ in which the product of the signs is $\varepsilon$, with $\varepsilon \in \{+,-\}$. Then, the associated maximal torus $(T_w)_\sigma$ of $G$ is contained in a subgroup $\SOT^{\varepsilon}_{2m}(q)$. If moreover $(T_w)_\sigma$ contains regular semisimple elements, then $(T_w)_\sigma$ is not contained in any subgroup $\SOT^{-\varepsilon}_{2m}(q)$.
\end{lemma}

\begin{proof}
Observe that $(T_w)_\sigma$ stabilizes a non-degenerate hyperplane $W$ of $\varepsilon$ sign (see \cite[Proposition 2.5.11]{KL}), so it is contained in a subgroup $\SO(W) \cong \SO^{\varepsilon}_{2m}(q)$. 
Moreover, if $(T_w)_\sigma$ contains regular semisimple elements, then $(T_w)_\sigma$ fixes only one nondegenerate hyperplane, and the last part of the statement follows.
\end{proof}

In the last part of the statement of Lemma \ref{stabilizers_hyperplanes_maximal torus}, it is necessary to assume that $(T_w)_\sigma$ contains regular semisimple elements. For instance, for $q=2$ the torus corresponding to $w=(1^+, \ldots, 1^+)$ is trivial -- hence contained in every subgroup of $G$.

We now prove two lemmas concerning extension field subgroups. 





\begin{lemma}
\label{extension_field_separable_dimension}
Assume $x \in \GLT_{n/b}(q^b) < \GLT_n(q)$ for some prime $b$ dividing $n$. Assume $x$ is separable over $\F_q$. Then $b$ divides the dimension of each irreducible constituent of the natural $\F_q[\gen x]$-module.
\end{lemma}

\begin{proof}
Assume $U$ is an irreducible space for $x$ over $\F_{q^b}$ of dimension $a$. Then $U$ has dimension $ab$ over $\F_q$. Moreover, $x$ acts homogeneously on $U$ over $\F_q$. Since $x$ is separable, it follows that $x$ acts irreducibly on $U$ over $\F_q$. The lemma follows.
\end{proof}

Now we aim to describe the maximal tori contained in extension field subgroups of maximal rank. In the case of unitary groups, recall that the subgroups of type $\GU_{n/b}(q^b) < \GU_n(q)$ only arise when $b$ is odd (cf. \cite[Section 4.3]{KL}), and that for $m$ odd $\GU_m(q)$ embeds in $\SO^-_{2m}(q)$, while for $m$ even it embeds in $\SO^+_{2m}(q)$. Also note that, as observed in Theorem \ref{classes_c1_c2_c3}, the extension field subgroups of type $\O_{n/b}(q^b) < \O^\varepsilon_n(q)$ with $n/b$ odd and $\varepsilon \in \{+,-,\circ\}$ are not of maximal rank, so we will not consider them.

As in (v) of Subsection \ref{subsec:maximal_tori}, for $G=\SO^+_{2m}(q)$ or $\Omega^+_{2m}(q)$ we denote by $(T_w)_\sigma$ any representative of a conjugacy class of maximal tori corresponding to a $W(D_m)$-class whose underlying signed partition is $w$.


\begin{lemma} Let $b$ be prime, and in each of the cases below, denote by $\sigma$ the defining Steinberg endomorphism.
\label{maximal_tori_extension_field_subgroup}
\begin{enumerate}
    \item The subgroup $\GLT_{n/b}(q^b)\cap \SLT_n(q) < \SLT_n(q)$ or $\GUT_{n/b}(q^b)\cap \SUT_n(q) < \SUT_n(q)$ contains a conjugate of $(T_w)_\sigma$, for every partition $w$ of $n$ in which each part has length divisible by $b$.
        \item Assume that $2m/b$ is even. The subgroup $\SpT_{2m/b}(q^b) < \SpT_{2m}(q)$ or \\ $\SOT^{\pm}_{2m/b}(q^b) < \SOT^{\pm}_{2m}(q)$ ($q$ odd) or $\Omega^{\pm}_{2m/b}(q^b) < \Omega^{\pm}_{2m}(q)$ ($q$ even) contains a conjugate of $(T_w)_\sigma$, for every signed partition $w$ of $m$ in which all parts have length divisible by $b$.
    \item The subgroup $\GUT_m(q) < \SpT_{2m}(q)$ or $\GUT_m(q) < \SOT^{\pm}_{2m}(q)$ ($q$ odd) or \\ $\GUT_m(q) < \Omega^{\pm}_{2m}(q)$ ($q$ even) contains a conjugate of $(T_w)_\sigma$, for every signed partition $w$ of $m$ in which odd parts have minus sign and even parts have plus sign.
\end{enumerate}
If moreover $(T_w)_\sigma$ contains separable elements, then each embedding of $(T_w)_\sigma$ in extension field subgroups of maximal rank has been listed in (1), (2), (3).
\end{lemma}

\begin{proof}
Let $H$ be an extension field subgroup as in the statement. Each maximal
torus of the ambient group contained in $H$ is a maximal torus of $H$, hence it admits a description in terms of its Weyl group. It is then easy to establish
the lemma with the help of Lemma \ref{extension_field_separable_dimension}  (see \cite[Section 4.3]{KL} for a detailed
description of the various embeddings).

Let us give the details of (1) for $\SL_n(q)$. Let $H=\GL_{n/b}(q^b)\cap \SL_n(q) < \SL_n(q)$, let $w$ be a partition of $n$ and let $(T_w)_\sigma$ be the corresponding maximal torus of $\SL_n(q)$. If all parts of $w$ have length divisible by $b$, then it is easy to see that $(T_w)_\sigma$ is contained in a subgroup conjugate to $H$.

Assume, on the other hand, that $(T_w)_\sigma$ is contained in a subgroup conjugate to $H$, and assume that $(T_w)_\sigma$ contains a separable element $g$; we want to show that all parts of $w$ have length divisible by $b$. Since $g\in H$, we see by Lemma \ref{extension_field_separable_dimension} that $b$ divides the dimension of each irreducible $\F_q[\langle x \rangle]$-submodule of $\F_q^n$. In particular, $g$ is contained in a torus $(T_{w'})_\sigma$ of $\SL_n(q)$, where all parts of $w'$ have length divisible by $b$. Since $g$ is separable and belongs to both $(T_w)_\sigma$ and $(T_{w'})_\sigma$, we deduce that $w=w'$, as required.
\end{proof}

\subsection{Proof of Theorem \ref{main_theorem_absolute_constant_2}}
\label{subsection_proof_classical_bounded_rank}

We are now ready to prove Theorem \ref{theorem_1.5_restatement} (hence Theorem \ref{main_theorem_absolute_constant_2}) for classical groups of bounded rank.


Let us prove the statement for unitary groups, for symplectic groups, and for $\Omega^+_{2m}(q)$ with $m$ even. The other cases are dealt with similarly. (In fact, we are omitting the proof of the easiest cases.) 

Since $q\geq Cr^r$ for some large constant $C$, it is easy to see that each maximal torus of $X_q$ contains separable elements, so Lemmas \ref{stabilizers_hyperplanes_maximal torus} and \ref{maximal_tori_extension_field_subgroup} apply.

For $x\in X_\sigma$, we denote by $\ca M_{\text{con}}(x)$ the members of $\ca M_{\text{con}}$ containing $x$ (this notation is used only here). Moreover, referring to the defining Steinberg endomorphism, we denote by $N^\pm_\ell$ the fixed points of the connected component of the stabilizer of a nondegenerate subspace of sign $\pm$ and of dimension $\ell$, and by $P_\ell$ the stabilizer of a totally singular $\ell$-space. Finally, in the following discussion we identify each torus with the corresponding (signed) partition; hence a partition can be contained in a subgroup.

\textbf{(i)} $\SU_n(q)$. If $n$ is odd, $\ca M_{\text{con}}(x_1)$ consists only of (unitary) extension field subgroups of type $\GU_{n/b}(q^b)$ ($b$ odd). On the other hand $\ca M_{\text{con}}(x_2)$ consists of $P_{(n-1)/2}$ and $N_1$. Therefore $\ca T(x_2)$ contains only the partitions with a $1$-cycle. By Lemma \ref{maximal_tori_extension_field_subgroup}, none of these belongs to $\ca T(x_1)$, hence $\ca T(x_1)\cap \ca T(x_2)=\emptyset$. Assume now $n$ is even. Again $\ca T(x_2)$ contains the partitions with a $1$-cycle. Moreover $\ca M_{\text{con}}(x_1)$ consists of unitary extension field subgroups, and $P_{n/2}$. Then $\ca T(x_1)\cap \ca T(x_2)=\emptyset$ (note that $P_{n/2}$ contains the partitions with all even parts).

\textbf{(ii)} $\Sp_{2m}(q)$. If $m$ is even and $q$ is odd, then $\ca M_{\text{con}}(x_1)$ contains every $\Sp_{2m/b}(q^b)$, but not $\GU_m(q)$ (by Lemma \ref{maximal_tori_extension_field_subgroup}). On the other hand $\ca M_{\text{con}}(x_2)$ consists only of $P_1$ and $N_2$. Hence $\ca T(x_2)$ contains the signed partitions with a $1$-cycle. By Lemma \ref{maximal_tori_extension_field_subgroup}, none of these partitions belongs to $\ca T(x_1)$, hence $\ca T(x_1)\cap \ca T(x_2)=\emptyset$.

If $m$ is odd and $q$ is odd, the difference is that $\ca M_{\text{con}}(x_1)$ contains $\GU_m(q)$, and as before every $\Sp_{2m/b}(q^b)$ (here $b$ is odd). On the other hand $\ca M_{\text{con}}(x_3)$ contains $P_m$, every $\Sp_{2m/b}(q^b)$, and does not contain $\GU_m(q)$. Moreover $\ca M_{\text{con}}(x_2)$ contains only $N_2$. Then $\ca T(x_2)$ contains signed partition with a $1$-cycle. But $(1^-, \ldots)$ is not contained in $\ca T(x_3)$, and $(1^+, \ldots)$ is not contained in $\ca T(x_1)$ (note that a partition contained in $P_m$ has all positive cycles). Hence $\ca T(x_1)\cap \ca T(x_2)\cap \ca T(x_3)=\emptyset$.

Assume now $q$ is even. 
Note that $\GU_m(q)$ is contained in $\Omega^+_{2m}(q)$ or in $\Omega^-_{2m}(q)$ according to whether $m$ is even or odd (cf. \cite[Section 4.3]{KL}). Moreover, by Lemma \ref{stabilizers_hyperplanes_maximal torus} we deduce that a partition contained in $P_m$ is contained in $\Omega^+_{2m}(q)$. As a consequence we can ignore $\GU_m(q)$ and $P_m$. The result follows with arguments as above. Indeed, $\ca T(x_2)\cap \ca T(x_3)$ consists of the signed partitions $w$ with a $1$-cycle. If the product of the signs of $w$ is $1$ (resp. $-1$), then $w$ does not belong to $\ca T(x_1)$ (resp. $\ca T(x_4)$). Therefore $\ca T(x_1)\cap \ca T(x_2)\cap \ca T(x_3)\cap \ca T(x_4)=\emptyset$.

\textbf{(iii)} $\Omega_{2m}^+(q)$ with $m$ even. As in (v) of Subsection \ref{subsec:maximal_tori}, we only look at the underlying signed partition of a $W(D_m)$-class (although, in fact, all the $W(D_m)$-classes that will be mentioned are $W(B_m)$-classes). We see that $\ca T(x_4)$ contains $(1^+, \ldots)$, $(1^-, \ldots)$, $(2^-, \ldots)$. Now $\ca M_{\text{con}}(x_3)$ consists of $N_4^-$ and subgroups of type $\Omega^+_m(q^2)$. Of the three subpartitions listed above, only $(2^-, \ldots)$ can belong to $\Omega^+_m(q^2)$. Therefore, $\ca T(x_3)\cap \ca T(x_4)$ contains $(1^+, 1^-, \ldots)$ and $(2^-, \ldots)$. At this point we note that $\ca M_{\text{con}}(x_1)$ contains $P_m$, $\GU_m(q)$ and every $\Omega^+_{2m/b}(q^b)$. Then $\ca T(x_1) \cap \ca T(x_3)\cap \ca T(x_4)$ consists of the partitions contained in $\Omega^+_m(q^2)$ and of type $(2^-, \ldots)$. We observe that none of these belongs to $\ca T(x_2)$, and the proof is concluded.

We summarize the fact that the proof of Theorem \ref{main_theorem_absolute_constant_2}(1) is complete.

\begin{theorem}
\label{t:1.5_bounded_rank}
    The conclusion to Theorem \ref{main_theorem_absolute_constant_2}(1) holds.
\end{theorem}

\begin{proof}
For $G\cong \PSL_2(q)$, see Theorem \ref{main_psl_2}. In the other cases, the statement follows from Theorem \ref{theorem_1.5_restatement}, which was proved in Section \ref{subsection_exceptional_groups} for exceptional groups, and in this section for classical groups. 
\end{proof}

\section{Proof of Theorems \ref{main_theorem_strengthening_bounded_rank} and \ref{main_theorem_two_random_elements_bounded_rank}}
\label{section_proof_theorem_1.2}

This section is devoted to proving Theorems \ref{main_theorem_strengthening_bounded_rank} and \ref{main_theorem_two_random_elements_bounded_rank}, which, together with Theorem \ref{theoremg2order}(i), imply Theorem \ref{main_theorem_bounded_away_zero} for groups of Lie type of bounded rank. By Theorem \ref{main_psl_2}, we may assume $r\geq 2$, where $r$ denotes the rank of the ambient simple algebraic group $X$.


Let $X$ and $\sigma$ be chosen as in Sections \ref{subsection_exceptional_groups} and \ref{subsection_classical_groups}. As we did for Theorem \ref{main_theorem_absolute_constant_2}, it is convenient for us to prove a slightly different version of Theorem~\ref{main_theorem_strengthening_bounded_rank}. In the following proof, for a subset $Y=\{x_1, \ldots, x_t\}$ of $X_q$, we will say that $Y$ invariably generates at least $(X_q)'$ if for every $g_1, \ldots, g_t \in X_q$, every maximal overgroup of $\gen{x_1^{g_1}, \ldots, x_t^{g_t}}$ in $X_q$ contains $(X_q)'$. Moreover, set
\[
\alpha=\alpha(X)= \begin{cases*}
1/|W(E_8)| & if $X$ is exceptional \\
1/(4r) & if $X$ is classical.
\end{cases*}
\]


\begin{theorem}
\label{main_theorem_strengthening_isogeny_types}
Assume $r\geq 2$ and $q\geq Cr^r$ for some large constant $C$. Assume $X_q\not\cong G_2(q)$ when $3\mid q$.
Then $\P^*_{\emph{inv}}(X_q,x)\geq \alpha + O(r^r/q)$ for a proportion of elements $x\in X_q$ of the form $1-O(r^r/q)$. 
\end{theorem}

\begin{proof}
Let $A_b$ be the set of elements of $(X_q)'$ listed in Tables~\ref{table_exceptional_groups} and~\ref{table_bounded_rank}. Let $x\in \Delta$, recalling the definition of $\Delta$ in Notation~\ref{notation}(x). By Remark \ref{theorem_consequence_proof} and Theorem \ref{theorem_1.5_restatement}, we deduce that $x$ invariably generates at least $(X_q)'$ with some $y\in A_b$. By our choice of the set $A_b$ in the various cases, $y$ is regular semisimple: let $S=(T_w)_\sigma$, with $w\in W$, be its maximal torus in $X_q$.

By the choice of $y$, and by the definition of $\Delta_S$, we see that $\{x,z\}$ invariably generates at least $(X_q)'$ for every $z \in \Delta_S$.



In order to conclude the proof, we only need to establish an appropriate lower bound on the proportions $|\Delta|/|X_q|$ and $|\Delta_S|/|X_q|$. We know that $|\Delta|/|X_q|=1-O(r^r/q)$ by Theorem \ref{almost_all_connected_components_maximal_rank}. Moreover $|\Delta_S|/|X_q|=\P(W,\sigma,w)+O(r^r/q)$ by Theorems \ref{almost_all_connected_components_maximal_rank} and \ref{correspondence_tori_weyl_group}. Clearly $\P(W,\sigma,w)\geq 1/|W|$, which is at least $1/|W(E_8)|$ for exceptional groups. For classical groups, one can check that $\P(W,\sigma,w)\geq 1/(4r)$ ,
which is attained for $(2^-,2^-)$ in $W(D_4)$ and for $(1^-, 1^-)$ in $W(C_2)$ (see Remark \ref{rem:weyl_bm}, below). 
The proof is concluded.
\end{proof}

We note that, by the same proof, with more care one can improve the value of $\alpha$ in case $X$ is exceptional.

\begin{remark}
\label{rem:weyl_bm}

We make some observations regarding the conjugacy class of a random element in wreath products which, in particular, readily imply the bound $\P(W,\sigma,w)\geq 1/(4r)$ stated in the last paragraph of the proof of Theorem \ref{main_theorem_strengthening_isogeny_types}.

Let $H$ be a finite group and let $G=H\wr S_m$. Recall by Remark \ref{rem:conjugacy_classes_wreath_products} that conjugacy classes in $G$ are in bijection with partitions of $m$ where each part is labelled by a conjugacy class of $H$. Now, a random element of $H\wr S_m$ is clearly given by a random permutation $\tau\in S_m$ and a random element $(x_1, \ldots, x_m)\in H^m$. It follows immediately from the first paragraph of Remark \ref{rem:conjugacy_classes_wreath_products} that the conjugacy class of a random element of $H\wr S_m$ can be computed in two steps: first draw a random permutation $\tau$ of $S_m$, and then label each cycle of $\tau$ by a conjugacy class $C$ of $H$ with probability $|C|/|H|$ (this gives naturally a labelled partition, just by viewing cycles as parts). 

In particular, in the case $H=C_2$, so $G=C_2\wr S_m\cong W(B_m)$, the conjugacy class of a random element of $G$ can be computed by drawing a random permutation $\tau$ of $S_m$, and equipping each cycle of $\tau$ with a sign $\pm$ with equal probability. For example, if $C$ is the class corresponding to $(m^+)$, then $|C|/|G|=1/(2m)$.

Consider now the subgroup $K\cong W(D_m)$ of $G\cong W(B_m)$ consisting of the elements $(x_1, \ldots, x_m)\tau$ with $x_i=\pm$ and $\prod_{i=1}^m x_i = +$. Then, the $G$-classes of elements of $K$ are in bijection with signed partitions of $m$ where the product of the signs is $+$. For a $K$-class $D$ whose $G$-class is $C$, we clearly have $|D|/|K| = \delta |C|/|G|$ where $\delta=2$ if $D=C$ and $\delta=1$ otherwise. For example, if $m$ is even and if $D$ is a $K$-class whose $G$-class $C$ corresponds to $(m^+)$, then  $|D|/|K|=1/(2m)$.


Finally, the $K$-classes of elements in $G\setminus K$ are $G$-stable, and are in bijection with signed partitions of $m$ where the product of the signs is $-$. 

From this discussion, we can easily deduce the bound $\P(W,\sigma,w) \geq 1/(4r)$, stated in the last paragraph of the proof of Theorem \ref{main_theorem_strengthening_isogeny_types}, which is attained for $(2^-,2^-)$ in $W(D_4)$ and for $(1^-, 1^-)$ in $W(C_2)$. Note that there are other cases close to this bound; for instance, $((m-1)^\pm, 1^\pm)$ in $W(C_m)$ for $m\geq 3$ gives probability $1/(4(m-1))=1/(4(r-1))$.

\end{remark}

We now deduce Theorem \ref{main_theorem_strengthening_bounded_rank} from Theorem \ref{main_theorem_strengthening_isogeny_types}.
\begin{proof}[Proof of Theorem \ref{main_theorem_strengthening_bounded_rank}]
By Theorem \ref{main_psl_2}, we may assume $r\geq 2$. With notation as in the proof of Theorem \ref{main_theorem_strengthening_isogeny_types}, we have that for every $x\in \Delta':=\Delta\cap (X_q)'$, $x$ invariably generates $(X_q)'$ with every element of $\Delta'_S:=\Delta_S\cap (X_q)'$. 
Therefore we only need to establish an appropriate lower bound on $|\Delta'|/|(X_q)'|$ and $|\Delta'_S|/|(X_q)'|$. By our choice of the type of $X$, the index of $(X_q)'$ in $X_q$ is bounded (it is at most $3$), hence $|\Delta'|/|(X_q)'|=1-O(r^r/q)$. Moreover, with reasoning similar to that in Lemma \ref{isogeny_types}, and using Theorem \ref{correspondence_tori_weyl_group}, we see that \[|\cup_{g\in X_q}(S\cap (X_q)')^g|/|(X_q)'| \geq \P(W,\sigma, w)+O(1/q),\] 
whence \[|\Delta'_S|/|(X_q)'| \geq \alpha +O(r^r/q).\]
This concludes the proof.
\end{proof}

\begin{proof}[Proof of Theorem~\ref{main_theorem_two_random_elements_bounded_rank}]
For the group $G=G_2(3^a)$, Theorem \ref{main_theorem_two_random_elements_bounded_rank} follows from Theorem \ref{theoremg2order}(i). 
The remaining cases follow from Theorem~\ref{main_theorem_strengthening_bounded_rank}.
\end{proof}

\subsection{An elaboration on Theorem \ref{main_theorem_two_random_elements_bounded_rank}}
\label{subsection_elaboration}
We obtain here an asymptotic for the probability in Theorem \ref{main_theorem_two_random_elements_bounded_rank}. For simplicity, we take $X$ of simply connected type (so $X_\sigma=X_q$ is quasisimple). Let $\{T_1, \ldots, T_\ell\}$ be a set of representatives for the $X_q$-conjugacy classes of maximal tori of $X_q$. Write $T_i=(T_{w_i})_\sigma$ with $w_i \in W$, as we did in Subsection \ref{subsection_weyl_group_tori}. Define a relation $\sim$ on $\{1, \ldots, \ell\}$ as follows. If $1 \leq i,j \leq \ell$, then $i \sim j$ if there are no conjugates of $T_i$ and $T_j$ with a common overgroup in $\ca M_{\text{con}}$.

\begin{theorem}
\label{theorem_elaboration} Assume $r\geq 2$. Let $x_1,x_2 \in X_q$ be chosen uniformly at random. Then,
\[
\P(\gen{x_1,x_2}_I=X_q)=\sum_{\substack{(i,j)\\i\sim j}} \P(W,\sigma,w_i)\P(W,\sigma, w_j) + \frac{d(r)}{q}
\]
for some function $d(r)$.
\end{theorem}

\begin{proof}
Note that $\Delta=\cup_{i=1}^\ell \Delta_{T_i}$, a disjoint union. By the definition of $\Delta$ and by Lemma \ref{semisimple_in_stable_torus}, whenever $i\sim j$ every element of $\Delta_{T_i}$ invariably generates $X_q$ with every element of $\Delta_{T_j}$. Clearly this is not true if $i\not\sim j$. By Theorems \ref{almost_all_connected_components_maximal_rank} and \ref{correspondence_tori_weyl_group}, we have $|\Delta_{T_i}|/|X_q|=\P(W,\sigma, w_i)+O(r^r/q)$. 
The statement follows. 
\end{proof}

\begin{remark}
We have a nice and rather explicit expression for the main term of $\P(\gen{x_1,x_2}_I=X_q)$, which one should be able to estimate with accuracy (for all exceptional groups it should be possible to compute the exact value). 
For instance, with notation as in Subsections \ref{subsection_G2_not_3} and \ref{subsection_G2_power_3}, in case $G=G_2(3^a)$ we have $3\sim 6$, $4\sim 5$ and $5\sim 6$. 
One deduces easily that $\P(\gen{x_1,x_2}_I=G)=1/9+O(1/q)$. A very easy case is $\SL_2(q)$, where the probability is $1/2+O(1/q)$ (for the error term, see the proof of Theorem \ref{main_psl_2} at the beginning of Section \ref{subsection_classical_groups}).

Note, also, that this is quite an unusual way to address a problem of random generation. Indeed, in these problems one usually proves that there is a small chance to be trapped in a maximal subgroup -- and, as a consequence, there is a large probability to generate. Here, on the other hand, we are directly \textit{exhibiting} many pairs of elements which (invariably) generate, which is a sort of opposite approach. In the language of Subsection \ref{subsection_context}, we are exhibiting large complete bipartite subgraphs of the graph $\Lambda_e(X_q)$.
\end{remark}

\section{A lower bound on $|A_b|$}
\label{section_lower_bound_|A|}
We show that there are cases in which we need $|A_b|\geq 4$ in Theorem \ref{main_theorem_absolute_constant_2} (note that we used a set of size at least $5$ only in the case $G=F_4(2^a)$).

\begin{lemma}
\label{four_elements_might_be_needed}
Assume $q$ is even and $m\geq 2$. Let $G=\emph{PSp}_{2m}(q)=\SpT_{2m}(q)$, and let $Y$ be a subset of $G$ of size $3$. Then, $\P_{\emph{inv}}(G,Y)\leq 1-1/2^mm! +O(1/q)$.
\end{lemma}

\begin{proof}
Note that $2^m m! = |W(C_m)|$. Let $y_1, y_2, y_3$ be elements of $G$; we claim that $\Omega:=\cap \ca T(y_i)\neq \emptyset$. Assume we prove the claim, and assume $S \in \Omega$. Then, the elements lying in $\wt S$ contribute to $1-\P_{\text{inv}}(G,Y)$. By Theorem \ref{correspondence_tori_weyl_group}, the proportion of elements lying in $\wt S$ is at least $1/|W(C_m)|+O(1/q)$. Therefore it is sufficient to prove the claim.

Note that the maximal torus corresponding to $w=(1^+, \ldots, 1^+)$ is contained in every maximal subgroup from class $\ca C_1$. Therefore, if the $y_i$ all act reducibly then this torus belongs to $\Omega$.

Now note that the torus corresponding to $w=(1^+, \ldots, 1^+, 1^-)$ is contained in $\SO^-_{2m}(q)$, and in every subspace stabilizer except for the stabilizer of a totally singular $m$-space. Therefore, if none 
of the $y_i$ acts as $(m^+)$ (i.e., irreducibly on two complementary totally singular subspaces), then this torus belongs to $\Omega$.

By the previous two paragraphs, we may assume that $y_1$ acts irreducibly and that $y_2$ acts as $(m^+)$. Note that both $y_1$ and $y_2$ belong to $\Sp_{2m/b}(q^b)$ for every prime divisor $b$ of $m$. Assume now $y_3$ lies in $\SO^{\varepsilon}_{2m}(q)$, with $\varepsilon \in \{+,-\}$ (it is well known that every element belongs to such a subgroup; cf. \cite{Dye}). Now observe that $\Sp_{2m/b}(q^b)$ and $\SO^{\varepsilon}_{2m}(q)$ contain a common maximal torus: the maximal torus corresponding to $w=(m^{\varepsilon})$, for instance. This concludes the proof.
\end{proof}

Next, we show that for groups of bounded rank we cannot have $|A_b|=1$ in Theorem \ref{main_theorem_absolute_constant_2}.

\begin{lemma}
\label{classical_bounded_rank_must_have_at_least_two_elements}
Let $X$ be a simple linear algebraic group, $\sigma$ a Steinberg morphism, and $x \in X_{\sigma}$. Then, $x$ is contained in a subgroup of $X_{\sigma}=X_q$ of maximal rank. In particular, $\P^*_{\emph{inv}}(X_q,x)\leq 1 - 1/|W| + O(1/q)$.
\end{lemma}

\begin{proof}
Write $x=us$ for the Jordan decomposition into the unipotent part $u$ and the semisimple part $s$. Every parabolic subgroup of $X_\sigma$ contains a Sylow $p$-subgroup of $X_\sigma$; hence $u$ belongs to a conjugate of every parabolic subgroup. Moreover $\text Z(X_\sigma)$ is contained in every parabolic of $X_\sigma$. 
Therefore if $s\in \text Z(X_\sigma)$ we have that $x$ is contained in a parabolic of $X_\sigma$. Assume then $s\notin \text Z(X_\sigma)$. Then $x \in \C_X(s)<X$, which is $\sigma$-stable. By Lemma \ref{semisimple_in_stable_torus}, $s$ is contained in a $\sigma$-stable maximal torus $T$ of $X$, hence $T\leq \C_X(s)$ and $\C_X(s)$ is of maximal rank. The first part of the statement is proved. The last part follows from Theorem \ref{correspondence_tori_weyl_group} and the fact that, if $S\in \ca T(x)$, then the elements of $\wt S$ contribute to $1-\P^*_{\text{inv}}(X_q,x)$. 
\end{proof}

\section{Groups of Lie type of large rank}
\label{section_large_rank}
In this final section we prove Theorem \ref{main_theorem_bounded_away_zero} and Theorem \ref{main_theorem_absolute_constant_2} for groups of Lie type of large rank. We work with quasisimple groups $G$ rather than the simple quotients $G/\Z(G)$ (this makes no difference, since $\Z(G)$ is contained in every maximal subgroup of $G$). Specifically, consider the following groups:
\begin{equation}
\label{eq:quasi_quasi}
G=\SL_n(q), \SU_n(q), \Sp_n(q), \Omega^\varepsilon_n(q)
\end{equation}
with $\varepsilon\in\{+,-,\circ\}$. We will denote by $V$ the natural $n$-dimensional module for $G$. We may assume that $n$ is large.


Recall that, in the bounded rank case, for classical groups we could focus on Aschbacher's classes $\ca C_1, \ca C_2, \ca C_3$ (Theorem \ref{classes_c1_c2_c3}). 
In the large rank case, the same happens. If $G$ is a finite quasisimple classical group, denote by $\ca M'=\ca M'(G)$ the set of all maximal subgroups of $G$ contained in one of the classes $\ca C_1, \ca C_2$ or $\ca C_3$ (recall Convention \ref{convention_aschbacher}).

\begin{theorem}
\label{large_rank_classes_c1_c2_c3}
\cite[Theorem 7.7]{FG3} Let $G$ be a finite quasisimple classical group of untwisted Lie rank $r$ defined over $\F_q$. For $r$ sufficiently large, the proportion of elements of $G$ which lie in subgroups not belonging to $\ca M'$ is $O(q^{-r/3})$.
\end{theorem}

\subsection{General case}
\label{subsection_general_large_rank}
In large rank, if $G$ is not symplectic in even characteristic and not orthogonal in odd dimension, we can easily deduce Theorems \ref{main_theorem_bounded_away_zero} and \ref{main_theorem_absolute_constant_2}(2) from known results.

\begin{theorem}\label{Al_1_theorems_1_5} The conclusions to Theorems \ref{main_theorem_bounded_away_zero} and \ref{main_theorem_absolute_constant_2}(2) hold for the groups $\PSL_n(q)$, $\mathrm{PSU}_n(q)$, $\PSp_{2m}(q)$ with $q$ odd, and $\POm_{2m}^\pm(q)$.
\end{theorem}
\begin{proof}
In view of Theorem \ref{main_theorem_strengthening_bounded_rank}, we may assume that $n$ and $m$ are large. We will prove Theorem \ref{main_theorem_absolute_constant_2}(2) (in these cases) with $|A_\ell|=1$, which, for $n$ or $m$ large, clearly implies Theorem \ref{main_theorem_bounded_away_zero}.

We work with the corresponding groups $G$ as in \eqref{eq:quasi_quasi}. Let $x\in G$ be the element defined in \cite[Table II]{GK}, and set $A_\ell=\{x\}$. By the proof of \cite[Proposition 4.1]{GK} it follows that $x$ is contained in no irreducible maximal subgroup of $G$. Next, let $\Omega$ be the set of integers which occur as the dimension of a proper nonzero subspace of $V$ fixed by $x$. By \cite[Table II]{GK}, we see that $\Omega$ has very small size (bounded absolutely from above), and $\Omega$ contains only integers $\ell$ such that both $\ell$ and $n-\ell$ are comparable to $n$ (up to constants). By \cite[Theorems 2.2, 2.3, 2.4, 2.5]{FG4}, 
\[
\frac{|\mathscr M(x)|}{|G|}=O(n^{-0.005}).
\]
We deduce by Lemma \ref{equivalence_inv_gen_tildas} that $\P_{\text{inv}}(G,x)=1-O(n^{-0.005})$, which concludes the proof.
\end{proof}

Now we need to deal with the remaining cases. We devote one subsection to each. The difference, in these cases, is that every element belongs to maximal subgroups whose union of conjugates is large.

\subsection{Orthogonal groups in odd dimension} Here we assume $G=\Omega_{2m+1}(q)$ with $q$ odd. Let $R^+$ and $R^-$ denote the union of the stabilizers of hyperplanes of plus and minus sign, respectively, and let $P_1$ denote the union of the stabilizers of singular $1$-spaces. It is well known and easy that every element of $G$ has eigenvalue $1$ on the natural module; in particular, $G=R^+\cup R^- \cup P_1$. 
We can now establish Theorem \ref{main_theorem_bounded_away_zero} in this special case.

\begin{theorem}\label{orthogonal_odd_theorem_1}
The conclusion to Theorem \ref{main_theorem_bounded_away_zero} holds in the case $G=\POm_{2m+1}(q)$ with $q$ odd.
\end{theorem}
\begin{proof}
In view of Theorem \ref{main_theorem_strengthening_bounded_rank}, we may assume that $m$ is large. Let $x \in G=\Omega_{2m+1}(q)$ be as in \cite[Table II]{GK}. By the proof of \cite[Proposition 4.1]{GK} it follows that the only maximal overgroup of $x$ is the stabilizer of a nondegenerate hyperplane of minus type. It follows from \cite[Theorem 9.26]{FG2} that the proportion of elements of $G$ lying in $R^-$ is bounded away from $1$ absolutely (it is at most $0.93$ for $n$ sufficiently large). We conclude by Lemma \ref{equivalence_inv_gen_tildas}.
\end{proof}

 We will see in Lemma \ref{orthogonal_odd_|A|_at_least_2} that $\P_{\text{inv}}(G,x)$ remains bounded away from $1$ for every $x\in G$. Next we want to prove Theorem \ref{main_theorem_absolute_constant_2} in this case. We have a strong dichotomy between the cases $q$ fixed and $q\rightarrow \infty$.

Recall that, in orthogonal groups, regular semisimple elements might have eigenvalues of multiplicity greater than $1$ (i.e., they need not be separable). We recall a lemma which specifies how this can happen, whose proof is straightforward and left to the reader.

\begin{lemma}
\label{orthogonal_odd_dimension_regular_semisimple}
Assume $g \in G$ is regular semisimple. Then $g$ centralizes a nondegenerate $1$-space, and $\emph{dim} \, \emph C_V(g)=1$. Moreover, either $g$ fixes no other nondegenerate $1$-space, or it acts as $-1$ on a nondegenerate $2$-space, and fixes no other nondegenerate $1$-space.
\end{lemma}

We next prove that, if $q$ is large, with high probability an element is separable.

\begin{theorem}
\label{orthogonal_odd_dimension_hyperplanes}
Assume $m\geq 3$. The proportion of elements of $G=\Omega_{2m+1}(q)$ which are separable is larger than $1-6/q$. These elements fix only one nondegenerate hyperplane.
\end{theorem}

\begin{proof}
By Lemma \ref{orthogonal_odd_dimension_regular_semisimple}, a separable element fixes only one nondegenerate $1$-space, hence only one nondegenerate hyperplane.
Therefore we only need to prove the first part of the statement.

By \cite[Theorem 2.3]{GL}, the proportion of regular semisimple elements in $G$ is at least $1-2/(q-1)-2/(q-1)^2$. Since $2/(q-1)+2/(q-1)^2+1/(q-1)\leq 6/q$, by Lemma \ref{orthogonal_odd_dimension_regular_semisimple} we just need to prove that the proportion of elements which act as $-1$ on a nondegenerate $2$-space $W$ is at most $1/(q-1)$ (note that a regular semisimple element cannot have equivalent modules of dimension at least $2$). 
Let $E$ be the set of such elements.

For fixed $W$, there are at most $|\Omega_{2m-1}(q)|$ choices for the element; indeed, for fixed $W$, it is determined whether the restriction of the element to $W^\perp$ belongs to $\Omega(W^\perp)$ or to $\SO(W^\perp) \setminus \Omega(W^\perp)$.
Now we have to sum over all possible $W$'s. Since $2(q-1)=|\GO_2^+(q)|<|\GO_2^-(q)|=2(q+1)$, there are at most $2|\GO_{2m+1}(q)|/|\GO^+_2(q)| |\GO_{2m-1}(q)|$ choices. In particular, we deduce that
\[
|E| \leq \frac{2 |\Omega_{2m-1}(q)|\cdot|\GO_{2m+1}(q)|}{ |\GO^+_2(q)|\cdot |\GO_{2m-1}(q)|} = \frac{2|\Omega_{2m+1}(q)|}{|\GO^+_2(q)|}=\frac{|\Omega_{2m+1}(q)|}{q-1}. 
\]
The proof is finished.
\end{proof}

\begin{theorem}\label{odd_orthogonal_theorem_5}
The conclusion to Theorem~\ref{main_theorem_absolute_constant_2}(2) holds for the case $G=\POm_{2m+1}(q)$ with $q$ odd.
\end{theorem}
\begin{proof}
Let $x=x_1$ be as in \cite[Table II]{GK}, and let $x_2$ act on the space as $(m\oplus m)\perp 1$ and having order $(q^m-1)/2$. 
Set $A_\ell=\{x_1, x_2\}$.

We claim that $x_2$ does not lie in maximal subgroups from classes $\ca C_2$ and $\ca C_3$. If this is true, then by Theorem \ref{large_rank_classes_c1_c2_c3} and \cite[Theorem 2.5]{FG4}, the proportion of elements of $G$ lying in conjugates of overgroups of both $x_1$ and $x_2$ is $|R^+\cap R^-|/|G| + O(n^{-0.005})$. By Theorem \ref{orthogonal_odd_dimension_hyperplanes}, $|R^+\cap R^-|/|G|\leq 6/q$, therefore $\P_{\text{inv}}(G,A_\ell) \geq 1-6/q+O(n^{-0.005})$ by Lemma \ref{equivalence_inv_gen_tildas}.

Therefore it suffices to prove the claim. Assume first $x_2 \in \GL_{n/b}(q^b)\rtimes C_b$ for some prime $b$ dividing $n$. If $m$ is large then $(q^m-1)/2b > q^{m/2}-1$. In particular, $x_2^b\in \GL_{n/b}(q^b)$ fixes only $m$-spaces, and only one $1$-space, which is impossible. 

Assume now $x_2 \in \GL_k(q)\wr S_t$ with $n=kt$ and $t>1$. The element $x_2$ induces a permutation $\pi$ of $S_t$ having at most $3$ cycles; hence the order of $\pi$, say $\ell$, is at most $n^3$. Then $x_2^{\ell} \in \GL_k(q)^t$. Again, for $m$ large $x_2^{\ell}$ fixes $m$-spaces and a $1$-space. Provided $n>3$, this is impossible for an element of $\GL_k(q)^t$, and the proof is finished. 
\end{proof}

Next we want to prove Theorem \ref{main_theorem_absolute_constant_2}(3) in this case.
The key fact is that, for $q$ fixed, the proportion of regular semisimple elements acting as $-1$ on a nondegenerate $2$-space is bounded away from zero. We recall an important result. 

\begin{theorem}
\label{proportion_separable_in_omega}
In the limit as $m \rightarrow \infty$, the proportion of elements of $\Om_{2m+1}(q)$ which are separable is equal to the proportion of elements of $\SOT_{2m+1}(q)\setminus \Omega_{2m+1}(q)$ which are separable. This proportion is at least $0.348$ for $m$ sufficiently large.
\end{theorem}

\begin{proof}
See \cite[Theorems 7.19 and 7.24]{FG2}.
\end{proof}

\begin{theorem}
\label{odd_dimension_orthogonal_-1_on_a_two_space}
If $m$ is sufficiently large, the proportion of elements of $G$ which are regular semisimple, and which act as $-1$ on a nondegenerate $2$-space of plus type, is at least $1/6q$. These elements fix hyperplanes of both signs, and fix a singular $1$-space.
\end{theorem}

\begin{proof}
The last statement is clear, since a $2$-space of plus type contains a singular $1$-space, and contains nondegenerate $1$-spaces of square and non-square discriminant. 
Therefore we only need to prove the first part of the statement. The proof is similar (although opposite in spirit) to Theorem \ref{orthogonal_odd_dimension_hyperplanes}. Let $E$ be the set of regular semisimple elements which act as $-1$ on a nondegenerate $2$-space of plus type $W$. For fixed $W$, by Theorem \ref{proportion_separable_in_omega} there at least $|\Omega_{2m-1}(q)|/3$ choices for the element on $W^\perp$. Then we have to sum through all $W$'s.  
We have
\[
|E| \geq \frac{|\Omega_{2m-1}(q)|}{3}\frac{|\GO_{2m+1}(q)|}{ |\GO^+_2(q)| \cdot |\GO_{2m-1}(q)|} = \frac{|\Omega_{2m+1}(q)|}{3|\GO^+_2(q)|}\geq \frac{|\Omega_{2m+1}(q)|}{6q},
\]
which concludes the proof. (Conceptually, there is nothing special here in considering a $2$-space: the same argument applies to elements acting as $-1$ on a space of bounded dimension.)
\end{proof}

\begin{theorem}\label{odd_orthogonal_theorem_52}
The conclusion to Theorem~\ref{main_theorem_absolute_constant_2}(3) holds for the case $G=\POm_{2m+1}(q)$ with $q$ odd.
\end{theorem}
\begin{proof}
We already observed that $G=R^+\cup R^- \cup P_1$. By Theorem \ref{odd_dimension_orthogonal_-1_on_a_two_space}, we have $|R^+\cap R^-\cap P_1|/|G|\geq 1/6q$ for sufficiently large $m$, hence $\P_{\text{inv}}(G,G)\leq 1-1/6q$ by Lemma \ref{equivalence_inv_gen_tildas}, which concludes the proof.
\end{proof}

We finally observe that we cannot have $|A_\ell|=1$ in Theorem \ref{main_theorem_absolute_constant_2}(2) (not even for $q\rightarrow \infty$).
\begin{lemma}
\label{orthogonal_odd_|A|_at_least_2}
Let $\varepsilon \in \{+,-\}$. Then, the proportions $|R^\varepsilon|/|G|$ and $|P_1|/|G|$ are at least $\delta_1$ for some absolute constant $\delta_1>0$. In particular, for every $x\in G$, $\P_{\emph{inv}}(G,x)\leq 1-\delta_1$.
\end{lemma}

\begin{proof}
If we prove the first part of the statement, the last part will follow from Lemma \ref{equivalence_inv_gen_tildas} and the fact that $G=R^+\cup R^-\cup P_1$. Therefore we only need to prove the first part. For $q$ fixed, we proved a stronger statement in Theorem \ref{odd_dimension_orthogonal_-1_on_a_two_space}. Next we deal with large $q$. By Theorem \ref{orthogonal_odd_dimension_hyperplanes}, the proportion of separable elements in $G$ is $1-O(1/q)$. If a separable element $g$ fixes a nondegenerate hyperplane $W$, then the maximal torus of $g$ is contained in the stabilizer of $W$ (indeed in a subgroup $\SO(W)$ of the stabilizer). 
The same is certainly true for the stabilizer of a singular $1$-space, since this is obtained as the fixed points of a connected subgroup of the algebraic group. Therefore, using Theorem \ref{correspondence_tori_weyl_group}, 
we see that the proportion of elements belonging to $R^\varepsilon$ (resp. $P_1$) is equal to $O(1/q)$ plus the proportion of elements of the Weyl group $W(B_m)\cong C_2\wr S_m$ corresponding to maximal tori fixing a nondegenerate hyperplane of $\varepsilon$ sign (resp. a singular $1$-space). For a nondegenerate hyperplane of $\varepsilon$ sign, these are the elements $(x_1, \ldots, x_m)\tau$ with $\tau \in S_m$, $x_i=\pm$ and $\prod x_i = \varepsilon$ (see Remark \ref{rem:conjugacy_classes_wreath_products}, which explains how to identify the class of an element), so their proportion is exactly $1/2$. For a singular $1$-space, these are the elements $(x_1, \ldots, x_m)\tau$ with $j\tau = j$ and $x_j=+$ for some $j$, and their proportion is at least $(1-1/e)/2$ for sufficiently large $m$; this follows from the fact that, for random $(x_1, \ldots, x_m)\tau$, the probability that $\tau$ fixes a point tends to $1-1/e$ as $m \rightarrow \infty$, and the corresponding entry in $(x_1, \ldots, x_m)$ is $\pm$ with equal probability.
\end{proof}

\subsection{Symplectic groups in even characteristic}
\label{subsection_symplectic_even}
The strategy, and the arguments, are often similar to those of the previous subsection. 

As in Remark~\ref{rem: convention}, we view $G=\Sp_{2m}(q)\cong \SO_{2m+1}(q)$, with subgroups $\SO^{\pm}_{2m}(q)$ corresponding to stabilizers of nondegenerate hyperplanes of $V=\F_q^{2m+1}$ (i.e., complements of $V^{\perp}$) of plus or minus type. 

As in the previous subsection, we denote by $R^+$ and $R^-$ the union of the stabilizers of hyperplanes of plus and minus sign, respectively. It is well known that $G=R^+\cup R^-$ (cf. \cite{Dye}). We will see in Lemma \ref{symplectic_even_|A|_at_least_2} that, also in this case, $\P_{\text{inv}}(G,x)$ is bounded away from $1$ for every $x\in G$. We can prove Theorem \ref{main_theorem_bounded_away_zero} in this case.

\begin{theorem}\label{symplectic_even_theorem_1}
The conclusion to Theorem \ref{main_theorem_bounded_away_zero} holds in the case $G=\mathrm{PSp}_{2m}(q)$ with $q$ even.
\end{theorem}
\begin{proof}
In view of Theorem \ref{main_theorem_strengthening_bounded_rank}, we may assume that $m$ is large. Let $x \in G=\Sp_{2m}(q)$ be as in \cite[Table II]{GK}. The same argument given for the other classical groups in Subsection \ref{subsection_general_large_rank} applies, except that $x$ stabilizes a unique nondegenerate hyperplane of plus or minus type. Therefore
\[
\frac{|\mathscr M(x)|}{|G|}=\frac{|R^\pm|}{|G|}+O(n^{-0.005}).
\]
Since $|R^\pm|/|G|$ is bounded away from $1$ by \cite[Theorem 9.15]{FG2} (it is at most $0.86$ for $n$ sufficiently large), the statement follows from Lemma \ref{equivalence_inv_gen_tildas}.
\end{proof}

We have now completed the proof of Theorem~\ref{main_theorem_bounded_away_zero}, and we summarize this fact  here.

\begin{proof}[Proof of Theorem \ref{main_theorem_bounded_away_zero}]
For $G=A_n$, see Theorem~\ref{alternating_theorems_1_5}. For $G=G_2(3^a)$, see Theorem~\ref{theoremg2order}(i). For the remaining groups of Lie type of bounded rank, we proved the stronger Theorem \ref{main_theorem_strengthening_bounded_rank} in Section \ref{section_proof_theorem_1.2}. For groups of Lie type of large rank, see Theorems ~\ref{Al_1_theorems_1_5}, \ref{orthogonal_odd_theorem_1} and \ref{symplectic_even_theorem_1}.
\end{proof}

\begin{theorem}\label{even_symplectic_theorem_5}
The conclusion to Theorem \ref{main_theorem_absolute_constant_2}(2) holds in the case $G=\PSp_{2m}(q)$ with $q$ even.
\end{theorem}

\begin{proof} Let $x=x_1\in G=\Sp_{2m}(q)$ be as in Table \cite[Table II]{GK}. In case $m$ is odd, for convenience we modify $x_1$ as follows. If $m\equiv 1$ mod $4$, we choose $x_1$ acting on the symplectic module as $(m-1)/2 \perp (m+3)/2 \perp (m-1)$; and if $m\equiv 3$ mod $4$ we choose $x_1$ acting as $(m+1)/2 \perp (m-3)/2 \perp (m+1)$. We let $x_1$ have order $q^b+1$ on each block of dimension $2b$. Similarly to the proof of Theorem \ref{odd_orthogonal_theorem_5}, we can easily prove that $x_1$ does not belong to subgroups of classes $\ca C_2$ and $\ca C_3$ if $m$ is large. (Subgroups of class $\ca C_3$ are ruled out since the element has nondegenerate irreducible modules whose dimensions differ by $2$; recall Lemma \ref{maximal_tori_extension_field_subgroup}.)
In this way, our element $x_1$ belongs to $\SO^-_{2m}(q)$ in all cases, both for $m$ even and $m$ odd.

Next, define $x_2 \in G$ as follows: if $m$ is odd, it acts as $(m-1)\perp (m+1)$; if $m \equiv 0$ mod $4$, it acts as $(m-2)\perp (m+2)$; if $m \equiv 2$ mod $4$, it acts as $(m-4) \perp (m+4)$. Assume moreover $x_2$ has order $q^b+1$ on each block of dimension $2b$. Except for stabilizers of subspaces, the only maximal overgroup of $x_2$ is a conjugate of $\SO_{2m}^+(q)$ (see \cite[Lemma 6.2]{BH}; in fact, a simpler argument applies since we only need to consider classes $\ca C_2$ and $\ca C_3$). 

Set now $A_\ell=\{x_1, x_2\}$. By Theorem \ref{large_rank_classes_c1_c2_c3} and \cite[Theorem 2.4]{FG4}, the proportion of elements lying in conjugates of overgroups of both $x_1$ and $x_2$ is $|R^+ \cap R^-|/|G| + O(n^{-0.005})$. 

By \cite[Theorem 2.3]{GL}, the proportion of regular semisimple elements in $G$ is at least $1-6/q$. A regular semisimple element does not have eigenvalue $1$ on the symplectic module (or, in other words, centralizes only $V^\perp$ on the orthogonal module $V$). It follows that a regular semisimple element $g$ fixes only one nondegenerate hyperplane, namely $[g,V]$.
Then $|R^+ \cap R^-|/|G|\leq 6/q$, which shows that $\P_{\text{inv}}(G,A_\ell)\geq 1-6/q+O(n^{-0.005})$. 
\end{proof}

Next we prove Theorem \ref{main_theorem_absolute_constant_2}(3) in this case, which is the final remaining step in the proof of all the main results. 

Let $G=\Sp_{2m}(q)$ with $m\geq 2$ and $q$ even. We first observe that, if $g \in G$ is semisimple and centralizes a $2$-space, then $g$ fixes hyperplanes of both signs. In the case $q$ odd, we could exploit the discriminant to see this; here we use a different argument.

\begin{lemma}
\label{both_sign_hyperplanes}
Assume $g \in G$ is semisimple and $\emph{dim} \, \emph C_V(g) \geq 2$ on the orthogonal module. Then $g$ fixes nondegenerate hyperplanes of both signs.
\end{lemma}

\begin{proof}
Assume $V^{\perp} = \gen v$. Since every element of $\F_q$ is a square, by rescaling we may assume $Q(v)=1$. Assume now $g$ is semisimple and fixes a nondegenerate hyperplane $W$; we want to show that $g$ fixes also a hyperplane of opposite sign.

Since $V=W \perp V^{\perp}$, by assumption there exists $0 \neq e \in W$ such that $eg=e$. Write $W=\gen e \oplus T$, with $T$ fixed by $g$. Assume first $Q(e) \neq 0$. Consider $e':=Q(e)^{-1/2}e+v$. Clearly $e'g=e'$ and $Q(e')=0$. Moreover, $g$ fixes $W':=\gen{e'} \oplus T$, which is a complement of $V^{\perp}$, i.e., a nondegenerate hyperplane. If $W'$ has opposite sign with respect to $W$, the proof is finished. Hence, replacing $W$ by $W'$ and $e$ by $e'$, we may assume from the beginning that $Q(e)=0$.

Since $g$ is semisimple, $g$ centralizes a nondegenerate $2$-subspace $\gen{e,f}$ of $W$, where $Q(f)=0$ and $(e,f)=1$. Write now $W=\gen{e,f} \perp U$, with $U$ fixed by $g$. Pick $\xi \in \F_q$ such that the polynomial $X^2+X+\xi^2$ is irreducible over $\F_q$. Then set $e':=e+v$, $f':=f+\xi v$ and $W':=\gen{e',f'} \perp U$. A straightforward computation shows that $\gen{e',f'}$ is a nondegenerate anisotropic space, i.e., $Q(x) \neq 0$ for every $0 \neq x \in \gen{e',f'}$ (cf. \cite[p. 26]{KL}). It follows now from \cite[Propositions 2.5.3 and 2.5.11]{KL} that $W'$ has opposite sign with respect to $W$. This concludes the proof.
\end{proof}

\begin{theorem}\cite[Theorem 7.11]{FG2}
The proportion of elements in $\SpT_{2m}(q)$ which are regular semisimple is at least $0.283$ for $m$ sufficiently large.
\end{theorem}

\begin{theorem}
\label{symplectic_q_even_1_on_a_two_space}
If $m$ is sufficiently large, the proportion of elements of $G$ which act (on the symplectic module) as the identity on a nondegenerate $2$-space, and which are regular semisimple on the orthogonal complement, is at least $1/4q^3$. These elements fix nondegenerate hyperplanes of both signs.
\end{theorem}

\begin{proof}
The last part of the statement follows from Lemma \ref{both_sign_hyperplanes}. The first part is exactly the same as in Theorem \ref{odd_dimension_orthogonal_-1_on_a_two_space} (one essentially replaces $\Om$ by $\Sp$ throughout, and we use $|\Sp_2(q)|\leq q^3$). 
\end{proof}

\begin{theorem}\label{even_symplectic_theorem_52}
The conclusion to Theorem~\ref{main_theorem_absolute_constant_2}(3) holds in the case $G=\PSp_{2m}(q)$ with $q$ even.
\end{theorem}
\begin{proof}
We already recalled that $G=R^+\cup R^-$, and by Theorem \ref{symplectic_q_even_1_on_a_two_space} we have that $|R^+\cap R^-|/|G|\geq 1/4q^3$ for sufficiently large $m$. Therefore $\P_{\text{inv}}(G,G)\leq 1-1/4q^3$ by Lemma \ref{equivalence_inv_gen_tildas}.
\end{proof}

The proof of Theorem~\ref{main_theorem_absolute_constant_2} is now complete, and we summarize it here.

\begin{proof}[Proof of Theorem \ref{main_theorem_absolute_constant_2}]

For part (1), see Theorem \ref{t:1.5_bounded_rank}. For part (2), the case of alternating groups is completed in Theorem~\ref{alternating_theorems_1_5}, and for groups of Lie type see Theorems~\ref{Al_1_theorems_1_5}, \ref{odd_orthogonal_theorem_5} and \ref{even_symplectic_theorem_5}. For part (3), the case $G=G_2(3^a)$ follows from Theorem \ref{theoremg2order}(ii). When $G=\POm_{2m+1}(q)$ with $q$ odd and $m$ large, see Theorem~\ref{odd_orthogonal_theorem_52}. Finally, when $G=\PSp_{2m}(q)$ with $q$ even and $m$ large, see Theorem~\ref{even_symplectic_theorem_52}.
\end{proof}

Thus the proof of all the main results of the paper is complete. We conclude by showing that we cannot have $|A_\ell|=1$ in Theorem \ref{main_theorem_absolute_constant_2} for $G=\PSp_{2m}(q)$ with $q$ even.

\begin{lemma}
\label{symplectic_even_|A|_at_least_2}
Let $\varepsilon \in\{+,-\}$. Then, $|R^\varepsilon|/|G|\geq \delta_2$ for an absolute constant $\delta_2>0$. In particular, for every $x\in G$, $\P_{\emph{inv}}(G,x)\leq 1-\delta_2$.
\end{lemma}

\begin{proof}
The last part follows from the first, Lemma \ref{equivalence_inv_gen_tildas} and $G=R^+\cup R^-$. We now prove the first statement. For $q$ fixed, we proved a stronger statement in Theorem \ref{symplectic_q_even_1_on_a_two_space}. Next we deal with large $q$. By Theorem \ref{almost_all_elements_regular_semisimple}, the proportion of regular semisimple elements in $G$ is $1-O(1/q)$. If a regular semisimple element $g$ fixes a nondegenerate hyperplane $W$, then the maximal torus of $g$ is contained in the stabilizer of $W$. 
As in the proof of Lemma \ref{orthogonal_odd_|A|_at_least_2}, we deduce that the proportion of elements belonging to $R^\varepsilon$ is equal to $O(1/q)$ plus the proportion of elements of $W(B_m)$ with product of sign $\varepsilon$, which is $1/2$.
\end{proof}

\bibliography{references}
\bibliographystyle{plain}
\end{document}